\documentclass[leqno, 11pt]{amsart}
\usepackage{amsmath} \usepackage{amssymb,amsxtra,amsthm,amscd}
\usepackage{graphicx} \usepackage[margin=40mm]{geometry}
\usepackage{url} \usepackage[pdftex]{hyperref}
  \usepackage{blaom}
  \renewcommand{\paragraph}[1]{{\indent\em #1.}}
  \DeclareMathOperator{\pseud}{pseud}
  \usepackage{pdfsync}
\begin{document}
\title%
[Pseudogroups via pseudoactions]%
{Pseudogroups via pseudoactions:\\ Unifying local, global, and
  infinitesimal symmetry} \author{Anthony
  D.~Blaom}%\\ \mbox{} \\ {\sffamily DRAFT}}%
\date{\today}%
\address{E-mail: {\tt anthony.blaom@gmail.com}}%
\thispagestyle{empty}
\begin{abstract}
  A multiplicatively closed, horizontal foliation on a Lie groupoid
  may be viewed as a `pseudoaction' on the base manifold $M$. A
  pseudoaction generates a pseudogroup of transformations of $M$ in
  the same way an ordinary Lie group action generates a transformation
  group. Infinitesimalizing a pseudoaction, one obtains the action of
  a Lie algebra on $M$, possibly twisted. A global converse to Lie's
  third theorem proven here states that every twisted Lie algebra
  action is integrated by a pseudoaction. When the twisted Lie algebra
  action is complete it integrates to a twisted Lie group action,
  according to a generalization of Palais' global integrability
  theorem.
\end{abstract}
\maketitle
%% Dedication
\vspace{-\baselineskip}
{\small
\begin{center}
   {\em Dedicated to Richard W. Sharpe}
 \end{center}} \vspace{\baselineskip}%
% \tableofcontents
%%%%%%%%%%%%%%%%%%%%%%%%%%%%%%%%% Body starts
%%%%%%%%%%%%%%%%%%%%%%%%%%%%%%%%% here %%%%%%%%%%%%%%%%%%%%%%%
\section{Introduction}

% \subsection{Local versus global symmetry}\lab{unifying}

% Isometry groups are global objects which do not necessarily reflect
% the local symmetry enjoyed by a geometric structure. For example, the
% isometry group of the Euclidean plane ${\mathbb R}^2$ is
% three-dimensional, while that of the cylinder ${\mathbb R}^2/{\mathbb
%   Z}$ is only two-dimensional, despite the fact that the cylinder has
% just as much symmetry {\em locally} as the plane. Such a distinction
% may persist at the level of infinitesimal isometries: in the cited
% examples, the Lie algebras of Killing fields have the same dimension
% as the respective isometry groups. To account also for local symmetry,
% one considers instead {\em pseudogroups} of local isometries. But
% while isometry groups are simple geometric objects --- namely
% finite-dimensional smooth manifolds --- pseudogroups, as stand-alone
% objects, are less tractable.

\subsection{Pseudoactions}\lab{krd}
Unlike transformation groups, pseudogroups of transformations capture
simultaneously the phenomena of global and {\em local} symmetry. On
the other hand, a transformation group can be replaced by the action
of an abstract group which may have nice properties --- a Lie group in
the best scenario. Here we show how to extend the abstraction of group
actions to handle certain pseudogroups as well. These pseudogroups
include: (i) all {\em Lie} pseudogroups of finite type (both
transitive and intransitive) and hence the isometry pseudogroups of
suitably regular geometric structures of finite type; (ii) all
pseudogroups generated by the local flows of a smooth vector field;
and, more generally (iii) any pseudogroup generated by the
infinitesimal action of a finite-dimensional Lie algebra. Our
generalization of a group action on $M$, here called a {\df
  pseudoaction}, consists of a multiplicatively closed, horizontal
foliation on a Lie groupoid $G$ over $M$. Ordinary group actions are
recovered as action groupoids equipped with the canonical horizontal
foliation. Details appear in \S\ref{section2} below.

Differentiating a pseudoaction, one obtains a flat Cartan connection
on a Lie algebroid. Cartan connections on Lie algebroids, introduced
in \cite{Blaom_06} and studied further in \cite{Blaom_12,Blaom_13},
are related to the classical connections bearing the same name. By
definition, a Cartan connection on a Lie algebroid is just an ordinary
linear connection $\nabla $, suitably respecting the underlying Lie
algebroid structure. Every Lie algebroid over $M$ equipped with a flat
Cartan connection is the quotient by covering transformations of an
action algebroid over the universal cover of $M$ (see \S\ref{twot}
below and \cite{Blaom_13}). For this reason we call Lie algebroids
equipped with such a connection {\df twisted Lie algebra
  actions}. However, we emphasise that it is the Cartan connection,
and not the action, that is ordinarily manifest in practice. In
particular, this applies to the infinitesimal isometries of finite
type geometric structures \cite{Blaom_12} and, more generally, to
PDE's whose infinitesimal symmetries are of finite type.

As far as we know, pseudoactions were first introduced by Tang
\cite{Tang_06}, using the language of \'etale groupoids. We discovered
the notion independently, sketching a relationship between
pseudoactions and their infinitesimal analogues in \cite[Appendix
A]{Blaom_12}.

\subsection{Integrating infinitesimal symmetry}\lab{hec}
A very well-known
observation of Sophus Lie is sometimes
phrased as follows:
\begin{theorem}[Lie, c.\,1880 \cite{Hermann_75}]
  Differentiating the action of an $r$-dimensional group of
  transformations on $M$ at the identity, one obtains an
  $r$-dimensional Lie subalgebra of vector fields on $M$.
\end{theorem}
\noindent% 
The naive converse statement is false: Not every Lie algebra of vector
fields on $M$ is the infinitesimalization of a transformation group on
$M$. Several partial converse statements, and generalizations thereof,
are well-known. For example, Lie himself showed that every
$r$-dimensional subalgebra of vector fields is generated by a {\em
  local} Lie group of transformations on $M$; see e.g.,
\cite{Duistermaat_Kolk_00}. (Of course all Lie's groups were
local.) On the other hand, {\em any} Lie algebra
is isomorphic to the Lie algebra ${\mathfrak g}_0$ of right-invariant
vector fields of some Lie group $G_0$. This is the result nowadays
referred to as `Lie III' but is actually due to \'Elie Cartan
\cite{Cartan_30,Cartan_36}.\footnote{In this article `Lie I for
  groups' refers to the theorem that any Lie algebra integrated by a
  Lie group is in fact integrated by a simply-connected Lie group,
  unique up to isomorphism; `Lie II for groups' is the theorem that
  every morphism of Lie algebras ${\mathfrak g}_1 \rightarrow
  {\mathfrak g}_2$ is integrated by a unique Lie group morphism $G_1
  \rightarrow G_2$, where $G_1 $ is a simply-connected Lie group
  integrating ${\mathfrak g}_1$, and $G_2$ any Lie group integrating
  ${\mathfrak g}_2$.} For a beautiful generalization of Lie III to Lie
groupoids, see Crainic and Fernandes
\cite{Crainic_Fernandes_03,Crainic_Fernandes_11}.

With the abstraction of Lie groups and Lie algebras and their actions
in hand, one has the following partial converse to Lie's theorem, due
to Richard Palais:
\begin{theorem}[Palais' local integrability theorem \cite{Palais_57a}]
  Every action of an abstract Lie algebra ${\mathfrak g}_0$ on $M$ is
  integrated by a local action of the simply-connected Lie group $G_0$
  having ${\mathfrak g}_0$ as its Lie algebra.
\end{theorem}
\noindent%
In \S\ref{hgn} we offer a short and novel proof of Palais' theorem
using Lie groupoids. This can be compared to Palais' original proof or
a contemporary proof such as \cite[Theorem 6.5]{Michor_08}. An
interesting question, addressed by Kamber and Michor
\cite{Kamber_Michor_04} but not here, is whether the space $M$ can be
`completed' to larger space on which the action of $G_0$ becomes
global.

For the special case of complete Lie algebra actions, Palais has
already provided the following global result:
\begin{theorem}[Palais' global integrability theorem \cite{Palais_57a}]
  If ${\mathfrak g}_0 $ above acts by complete vector fields, then the
  action is integrated by a global action of $G_0$ on $M$.
\end{theorem}

\vspace{0.75\baselineskip}%
Now according to the constructions outlined in \S\ref{krd} above, we
have the following generalization of Lie's theorem cited above: {\it
  Differentiating a pseudoaction one obtains a twisted Lie algebra
  action.} The main preoccupation of the present paper will be to
prove the following {\em global} converse:
\begin{theorem}[Lie III for pseudoactions]
  Every twisted Lie algebra action is integrated by some pseudoaction.
\end{theorem}
\noindent%
Here we mean `integrated' in the strongest sense: the pseudogroup
defined by the pseudoaction, and the pseudogroup generated by local
flows of the infinitesimal generators of the Lie algebra action,
coincide. The integrating pseudoaction is uniquely determined, up to
isomorphism, if we insist the underlying groupoid has simply-connected
source fibres. See the restatement Theorem \ref{lie3} below for
details.

Under an appropriate completeness hypothesis we also obtain a
generalization of Palais' global integrability theorem:
\begin{theorem}
  Every complete twisted Lie algebra action is integrated by a twisted
  Lie group action.
\end{theorem}
\noindent%
Twisted Lie group actions are defined in \S\ref{tg}; completeness in
\S\ref{twisted}.

To complete this brief survey of old and new integrability results, we
note a recently published theorem which applies in the special case of
{\em transitive} actions:
\begin{theorem}[\cite{Blaom_13}]
  A smooth manifold $M$ admits a transitive, geometrically closed,
  twisted Lie algebra action if and only if it is locally homogeneous.
\end{theorem}
\noindent%
Recall here that $M$ is {\df locally homogeneous} if $M$ admits an
atlas of coordinate charts modelled on a homogeneous space $G/H$,
where the transition functions are right translations by elements of
$G$. The closedness condition (not defined here) is needed to ensure
$G/H$ is a Hausdorff manifold. Locally homogeneous manifolds are less
general than transitive pseudogroups generated by a pseudoaction, but
more general than transitive twisted Lie group actions (manifolds
homogeneous `up to cover', see op.~cit.).

\subsection{Paper outline}
Pseudoactions, and the pseudogroups they define, are described in
\S\ref{section2}, which includes a detailed formulation of the two new
results alluded to above, reformulated as Theorem \ref{lie3} --- what
we refer to as `Lie III for pseudoactions' --- and Theorem
\ref{twisted}, a generalization of Palais' global integrability
theorem to the twisted case. It is an immediate corollary of Theorems
\ref{lie3} and \ref{goofy} that all the pseudogroups on the list in
\S\ref{krd} above can be encoded by pseudoactions. The simplest
non-trivial pseudoactions, twisted Lie group actions, are also
described in \S\ref{section2}. As mentioned above and reviewed in
\S\ref{twot}, their infinitesimal counterparts amount to flat Cartan
connections on a Lie algebroid. We conclude \S\ref{section2} by
revisiting Olver's example of a non-associative local Lie group
\cite{Olver_96} from the pseudoaction point of view.

In \S\ref{av} we prove Lie III for pseudoactions in the untwisted
case. The proof is based on an explicit model for a Lie groupoid
integrating a given action algebroid, whose existence is guaranteed by
a theorem of Dazord \cite{Dazord_97}, but which differs from
Dazord's model (which seems unsuitable for our purposes). However, to
show our model is well-defined, we must appeal to Palais' local
integrability theorem, and our new proof of this theorem depends on
Dazord's existence result.

In order to generalize to the untwisted case, we first formulate and
prove versions of Lie I and Lie II for pseudoactions, which are of
independent interest; these are Theorems \ref{mjy} and \ref{mre}
respectively. Theorem \ref{mre} gives sufficient conditions for
integrating morphisms of twisted Lie algebra actions to morphisms of
pseudoactions, but is restricted to the case of morphisms covering a
local diffeomorphism of base manifolds (sufficient for present
purposes). We conjecture that Theorems \ref{mjy} and \ref{mre} hold
for general morphisms and suggest the use of Salazar's cocycle
techniques \cite{Salazar_13} to extend our proofs.

With Lie I and Lie II in hand, the untwisted version of Lie III
implies the general case, as we show in \S\ref{hes}. Finally, the case
of complete twisted Lie algebra actions is treated in \S\ref{onu}.

\vspace{0.7\baselineskip}%
We acknowledge Marius Crainic, Rui Fernandes, Bas Janssens, Erc\"ument
Orta\c{c}gil, Maria Amelia Salazar and Ori Yudilevich for helpful
discussions.\footnote{Part of the research reported here was carried
  out during the author's visit in 2014 to Utrecht University, whose
  support is gratefully acknowledged.}

\subsection*{Notation} Throughout this paper $G$ denotes a Lie
groupoid, $G_0$ a Lie group, ${\mathfrak g} $ a Lie algebroid, and
${\mathfrak g}_0$ a Lie algebra.

\section{Detailed formulation of main results}\lab{section2}
The reader is assumed to have some familiarity with Lie groupoids and
Lie algebroids --- say what is contained in
\cite{CannasdaSilva_Weinstein_99}. For more detail, see
\cite{Crainic_Fernandes_11,Mackenzie_03}.

\subsection{Pseudotransformations and pseudoactions}\lab{cause}
Let $G$ be a Lie groupoid over a smooth, connected, paracompact
manifold $M$, and call an immersed submanifold $\Sigma\subset G$ a
{\df pseudotransformation} if the restrictions to $\Sigma$ of the
groupoid's source and target maps are local diffeomorphisms.  For
example, any smooth $n$-dimensional submanifold of $M \times M$ that
is locally a graph over {\em both} projections $M \times M \rightarrow
M$ is a pseudotransformation of the pair groupoid $M \times M$. Thus
every local diffeomorphism of $M$ may be regarded as a
pseudotransformation, but so also may its `inverse'.
% For example, every locally-defined diffeomorphism of $M$ may be
% regarded as a pseudotransformation of the pair groupoid $M\times M$;
% so also is any `multi-valued transformation' provided it localizes
% everywhere to a bona fide diffeomorphism of open sets. (Think, for
% example, of the square-root function $z \mapsto z^{1/2}$ on the
% punctured complex plane $M={\mathbb C}\backslash\{0\}$.)

A {\df pseudoaction} on $G$ is any smooth foliation ${\mathcal F}$ on
$G$ such that:
\begin{conditions}
        \item  The leaves of ${\mathcal  F}$ are pseudotransformations.
        \item ${\mathcal  F}$ is  {\df multiplicatively closed}.\lab{potato}
\end{conditions}
To define what is meant in \eqref{potato}, regard local bisections of
$G$ as immersed submanifolds and let $\widehat{\mathcal F}$ denote the
collection of all local bisections that intersect each leaf of
${\mathcal F} $ in an {\em open} subset. Let $ \widehat G$ denote the
collection of {\em all} local bisections of $G$, this being a groupoid
over the collection of all open subsets of $M$.  Then condition
\eqref{potato} is the requirement that $\widehat{\mathcal F}\subset \widehat
G$ be a subgroupoid.

% In the sequel, all  pseudotransformations and pseudoactions are
% understood to be smooth. 

Given a pseudoaction ${\mathcal F}$ of $G$ on $M$, each element $b \in
\widehat{\mathcal F}$ defines a local transformation $\phi_b$ of $M$. We
let $\pseud({\mathcal F})$ denote the set of all local transformations
$\psi \colon U \rightarrow V$ of $M$ that are locally of this form;
that is, for every $m \in U$, there exists a neighbourhood $U' \subset
U$ of $m$ such that $\psi|_{U'}= \phi_b$ for some for some $b \in\widehat
{\mathcal F}$. Then, by \eqref{potato}, $\pseud({\mathcal F})$ is a
pseudogroup.

The basic prototype of a pseudoaction is any smooth action of a Lie
group $G_0$ on $M$: take as Lie groupoid the action groupoid $G:=G_0
\times M $ and as foliation ${\mathcal F}$ the one with leaves
$\{g\}\times M$, $g \in G_0$. In this case $\pseud({\mathcal F})$
consists of all local transformations $\phi \colon U \rightarrow V$ of
$M$ such that restrictions of $\phi $ to connected components of $U$
are of the form $m \mapsto g \cdot m$, for some $g \in G_0$.

% \subsection{Pseudoactions as \'etalifications}\lab{et}
% Tang \cite{Tang_06} describes pseudoactions using the language of
% \'etale groupoids. An {\df \'etalification} of a Lie groupoid $G$ over
% $M$ is a second Lie groupoid $G'$ with the same underlying set, the
% same base $M$, and the same source and target projections as $G$, and
% such that:
% \begin{conditions}
% \item $G'$ is an \'etale Lie groupoid (i.e., source and target maps
%   are local diffeomorphisms); and
% \item The tautological map $G' \rightarrow G $ is a morphism of Lie groupoids.
% \end{conditions}
% Evidently, as an immersed submanifold, the image of $G' \rightarrow
% G$ is the union of the leaves of a regular foliation ${\mathcal F}$ on
% $G$, and this foliation is a pseudoaction in our sense. Conversely,
% every pseudoaction ${\mathcal F}$ of $G$ is generated by some
% \'etalification $G' \rightarrow G$, unique up to isomorphism. In this
% language, $\pseud({\mathcal F})$ is simply the pseudogroup of local
% transformations defined by local bisections of $G'$.

% Note that we make no use of the preceding observations in the present
% article.

\subsection{Pseudoactions as flat Cartan connections}\lab{onepoint}
Let $G$ be a Lie groupoid and $J^1 G$ the corresponding Lie groupoid
of one-jets of local bisections of $G$. Then a right-inverse $S \colon
G \rightarrow J^1 G$ for the canonical projection $J^1 G \rightarrow
G$ determines a rank-$n$ distribution $D \subset TG$ on $G$, where
$n=\dimension M$. If $S \colon G \rightarrow J^1 G$ is additionally a
morphism of Lie groupoids, then we call $D$ (or $S$) a {\df Cartan
  connection} on $G$. We prove the following elementary observation in
\S\ref{terminology}:
\begin{proposition}
  An arbitrary foliation ${\mathcal F}$ on $G$ is a pseudoaction if
  and only if its tangent distribution $D$ is a Cartan connection on
  $G$.
\end{proposition}
\noindent% 

Now let ${\mathcal F}$ be a pseudoaction of $G$ and $S \colon G
\rightarrow J^1 G$ the corresponding Cartan
connection. Infinitesimalizing $S$, we obtain a splitting $s \colon
{\mathfrak g} \rightarrow J^1 {\mathfrak g} $ of the exact sequence
\begin{equation*}
  0\longrightarrow T^*\!M\otimes{\mathfrak g}\longrightarrow J^1{\mathfrak 
    g}\longrightarrow{\mathfrak g}\longrightarrow 0.
\end{equation*}
Here $J^1 {\mathfrak g} $ denotes the Lie algebroid of one-jets of
sections of the Lie algebroid ${\mathfrak g} $ of $G$. (A canonical
identification of the Lie algebra of $J^1 G$ with $J^1 {\mathfrak g} $
is recalled in \S\ref{jet} below.)  In the category of vector bundles,
splittings $s$ of the above sequence are in one-to-one correspondence
with linear connections $\nabla$ on ${\mathfrak g}$; this
correspondence\footnote{According to our sign conventions, the
  inclusion $T^*\!M \otimes {\mathfrak g} \rightarrow J^1 {\mathfrak
    g} $ induces the following map on sections: $df \otimes X \mapsto
  f J^1 X - J^1 (f X)$.} is given by
\begin{equation}
  sX = J^1 X + \nabla X; \qquad  X \in \Gamma({\mathfrak g}). \mathlab{cspan}
\end{equation}
When $s \colon {\mathfrak g}
\rightarrow J^1 {\mathfrak g} $ is a morphism of Lie algebroids, as is
the case here, then $\nabla $ is called an {\df (infinitesimal) Cartan
  connection} on $\mathfrak g$.

Certainly not all Cartan connections on ${\mathfrak g}$ arise from
smooth pseudoactions. Rather, one has the following fundamental
result:%, proven in \S\ref{lie0}:
\begin{theorem}
  Assume $G$ is source-connected. Then a Cartan connection $D$ is
  tangent to a pseudoaction if and only if the corresponding
  infinitesimal Cartan connection $\nabla$ is flat.
\end{theorem}
\noindent%
This theorem follows from a result proven recently by Salazar
\cite[Theorem 6.4.1]{Salazar_13} for arbitrary `multiplicatively
closed' distributions on a Lie groupoid, of which the distribution $D$
defined by a Cartan connection is a special case. An independently
obtained proof along different lines will be given elsewhere. Both
proofs involve techniques not needed in the rest of the paper.

In general, an arbitrary Cartan connection on a Lie groupoid $G$
defines a rank-$n$ distribution on $G$ that is not integrated by any
foliation ${\mathcal F} $.\newline

\subsection{Twisted Lie group actions}\lab{tg}
To describe the simplest non-trivial examples of pseudoactions, let
$G_0$ be a Lie group acting on the universal cover $\tilde M$ of
$M$. Assume the action respects the group of covering transformations
$\Lambda \cong \pi_1(M)$ in the following sense: If $G_0$ acts
effectively and is identified with a subgroup of $\diff(M)$, then
require that $\Lambda \subset \diff(M)$ be contained in the normaliser
of $G_0$. More generally, we require the existence of a group
homomorphism $\lambda \mapsto \nu_\lambda \colon \Lambda \rightarrow
\automorphism(G_0)$ such that
\begin{equation}
  \lambda(g \cdot  \tilde m) = \nu_\lambda(g)\cdot \lambda(\tilde m); \qquad
  \lambda \in \Lambda,\, g \in G_0,\, \tilde m \in \tilde M. \mathlab{starry}
\end{equation}
(We say that $\lambda \colon \tilde M \rightarrow \tilde M$ is
$G_0$-equivariant with twist $\nu_\lambda$.) This requirement can
always be satisfied when $G_0$ contains $\Lambda $ as a subgroup,
provided the action of $G_0$ on $\tilde M$ extends the tautological
action of $\Lambda$; in that case take $\nu_\lambda (g)=\lambda g
\lambda^{-1}$. In any case, assuming \eqref{starry} holds, $\Lambda $
acts on the action groupoid $G_0 \times \tilde M$ by Lie groupoid
automorphisms and the quotient $G_0 \times_\nu M := (G_0 \times \tilde
M)/\Lambda$ becomes a Lie groupoid over $M$. Moreover, the canonical
foliation of $G_0 \times M$ by copies of $M$ drops to a pseudoaction
${\mathcal F}$ on $G_0 \times_\nu M $.  Such pseudoactions will be
called {\df twisted Lie group actions}.

It is easy to see $\pseud({\mathcal F})$ is the pseudogroup generated
by all local transformations $\varphi \colon U \rightarrow V$ of $ M$
covered by a local transformation of $\tilde M$ of the form $\tilde m
\rightarrow g \cdot \tilde m$, $g \in G_0$.

\begin{example}
  Take $G_0$ to be the three-dimensional group of isometries of the
  Euclidean plane ${\mathbb R}^2$ to obtain a twisted Lie group action
  on the cylinder $S^1 \times {\mathbb R} \cong {\mathbb R}^2/{\mathbb
    Z}$. The associated pseudogroup of transformations is the
  pseudogroup of local and global isometries of $S^1 \times {\mathbb
    R} $. This pseudogroup includes the isometry {\em group} of $S^1
  \times {\mathbb R} $ (i.e., the global isometries) which is only
  two-dimensional.
\end{example}

\subsection{Twisted Lie algebra actions}\lab{twot}
According to Theorem \ref{onepoint}, the infinitesimalization of a
pseudoaction is a Lie algebroid ${\mathfrak g}$ equipped with a flat
Cartan connection $\nabla$. We call such a pair $({\mathfrak g},
\nabla)$ a {\df twisted Lie algebra action}. To justify this
terminology, let ${\mathfrak g}_0 $ be a finite-dimensional Lie
algebra acting smoothly from the left on the universal cover $\tilde
M$ of $M$, and denote the corresponding Lie algebra homomorphism
${\mathfrak g}_0 \rightarrow \Gamma (T \tilde M)$ by $\xi \mapsto
\xi^\dagger$. We assume this action
respects the group of covering transformations $\Lambda$ in the
following sense: there exists a representation $\mu \mapsto \mu_\lambda
\colon \Lambda \rightarrow \automorphism({\mathfrak g}_0)$ of
$\Lambda$ by Lie algebra automorphisms such that
\begin{equation*}
  \lambda_* \xi^\dagger = (\mu_\lambda \xi)^\dagger; \qquad \lambda \in
  \Lambda,\,\xi \in {\mathfrak g}_0.
\end{equation*}
Here $\lambda_*$ denotes pushforward. (This requirement is just the
infinitesimal form of \eqrefs{tg}{starry}; we say $\lambda \colon
\tilde M \rightarrow \tilde M $ is ${\mathfrak g}_0 $-equivariant with
twist $\mu_\lambda$.) Then the action of $\Lambda $ on the action
algebroid ${\mathfrak g}_0 \times \tilde M$, defined by $\lambda \cdot
(\xi,\tilde m)=(\mu_\lambda \xi,\lambda(\tilde m))$, is by Lie
algebroid automorphisms, and the quotient ${\mathfrak g}_0 \times_\mu
M:=({\mathfrak g}_0 \times \tilde M)/\Lambda$ becomes a Lie algebroid
over $M$. Moreover, the canonical flat connection on ${\mathfrak g}_0
\times \tilde M$, which is Cartan, drops to a flat Cartan connection
$\nabla $ on ${\mathfrak g}_0 \times_\mu M$.

Conversely, we have the following trivial observation, recorded for
later on:
\begin{remark}
  The formal pullback of ${\mathfrak g}_0 \times_\mu M $ to a Lie
  algebroid $\tilde {\mathfrak g} $ over $\tilde M$ is canonically
  isomorphic to ${\mathfrak g}_0 \times \tilde M$, with the canonical
  action of $\Lambda $ on $\tilde {\mathfrak g} $ being represented by
  the action of $\Lambda $ on ${\mathfrak g}_0 \times \tilde M$
  described above.
\end{remark}
\noindent%
The main point, however, is:
\begin{proposition}[\cite{Blaom_13}]
  Every Lie algebroid ${\mathfrak g} $ over $M$, equipped with a flat
  Cartan connection $\nabla $ --- i.e., every twisted Lie algebra
  action --- is naturally isomorphic to ${\mathfrak g}_0 \times_\mu
  M$, for some ${\mathfrak g}_0$ and $\mu$ as above.
\end{proposition}
\noindent%
In detail: Let $\tilde {\mathfrak g}$ denote the pullback of
${\mathfrak g}$ to a Lie algebroid over $\tilde M$, $\tilde \nabla $
the pullback connection, and ${\mathfrak g}_0$ the finite-dimensional
vector space of $\tilde \nabla$-parallel sections of $\tilde
{\mathfrak g}$. Then ${\mathfrak g}_0 \subset \Gamma(\tilde {\mathfrak
  g})$ is a Lie subalgebra acting on $\tilde M$ according to
$\xi^\dagger(\tilde m)=\#\xi (\tilde m)$ where $\#$ denotes anchor;
the monodromy $\mu $ of the flat connection $\nabla $ is a
representation of $\pi_1(M) \cong \Lambda$ on ${\mathfrak g}_0 $ by
Lie algebra automorphisms; and an isomorphism ${\mathfrak g}_0
\times_\mu M \cong {\mathfrak g}$ is given by
\[ (\xi, \tilde m)
\modulo \Lambda \mapsto \pi(\xi(\tilde m)),
\]
where $\pi \colon \tilde {\mathfrak g} \rightarrow {\mathfrak g} $ is
the canonical projection. For further details see \cite{Blaom_13}.

\subsection{Pseudogroups generated by infinitesimal data}\lab{hsu}
Let $({\mathfrak g}, \nabla)$ be a twisted Lie algebra action on $M$
and, for each {\em local} $\nabla $-parallel section $X$ of
${\mathfrak g}$, let $\Phi_{\# X}^t$ denote the corresponding time-$t$
flow map\footnote{We do not restrict $t$ but allow $\Phi_{\# X}^t$ to
  have empty domain, i.e., to be the `empty transformation', whose
  composition with any other is itself.} of the vector field
$\#X$. Here $\#\colon {\mathfrak g} \rightarrow TM$ denotes the anchor
of $\mathfrak g$. Then the collection all local transformations of $M$
of the form $\Phi_{\#X}^t$, for some such $X$ and $t \in {\mathbb R}$,
generates a pseudogroup denoted $\pseud(\nabla)$.  Locally, each
element of $\pseud(\nabla)$ is of the form
\[
\Phi_{\#X_1}^{t_1} \circ \Phi_{\#X_2}^{t_2}\circ\cdots \circ
\Phi_{\#X_k}^{t_k},
\]
for some locally defined sections $X_1,X_2,\ldots,X_k$ of $\mathfrak
g$ and $t_1,t_2,\ldots,t_k \in {\mathbb R}$. In the special case that
$({\mathfrak g}, \nabla )$ is an action algebroid, this is the usual
pseudogroup of transformations generated by flows of the infinitesimal
generators of the action.

\subsection{Integrability and formal integrability}\lab{lie3}
A pseudoaction ${\mathcal F} $ on a Lie groupoid $G$ over $M$ will be
said to {\df integrate} a twisted Lie algebra action $({\mathfrak
  g},\nabla)$ on $M$ if $\pseud({\mathcal F})=\pseud(\nabla)$.  We
will say that ${\mathcal F} $ {\df formally integrates} $({\mathfrak
  g},\nabla)$ if $G$ integrates ${\mathfrak g}$ --- i.e., ${\mathfrak
  g}$ is the Lie algebroid of $G$ --- and if $\nabla $ is the
infinitesimalization of the Cartan connection $D$ tangent to
${\mathcal F}$, as we have described in \S\ref{onepoint}. Note that an
integration ${\mathcal F}$ need not be a formal integration.

Here is our main result, proven in \S\ref{hes}:
\begin{theorem}[Lie III for pseudoactions]
  For every twisted Lie algebra action $({\mathfrak g}, \nabla)$ there
  exists a pseudoaction ${\mathcal F} $ that simultaneously integrates
  and formally integrates $({\mathfrak g}, \nabla)$. Moreover, any
  pseudoaction ${\mathcal F}$ on a Lie groupoid $G$ formally
  integrating $({\mathfrak g},\nabla)$ is a bona fide integration if
  $G$ has connected source-fibres.
\end{theorem}
\noindent% 
Furthermore, by Lie I for pseudoactions (Theorem \ref{mjy}),
${\mathcal F} $ is {\em uniquely} determined up to isomorphism, if we
insist that $G$ have simply-connected source fibres.

\subsection{Complete twisted Lie algebra actions}\lab{twisted}
In \S\ref{onu} we prove the following generalization of Palais'
global integrability theorem quoted in \S\ref{hec}, which strengthens the conclusion of the preceding
theorem under an additional hypothesis:
\begin{theorem}
  Every {\em\bf complete} twisted Lie algebra action $({\mathfrak g},
  \nabla)$ on $M$ is formally integrated (and hence integrated) by a
  twisted Lie group action.
\end{theorem}
\noindent%
In detail: Let $(\tilde {\mathfrak g}, \tilde \nabla $) denote the
pullback of $({\mathfrak g}, \nabla)$ to the universal cover $\tilde
M$, and ${\mathfrak g}_0 \subset \Lambda(\tilde {\mathfrak g})$ the
Lie subalgebra of $\tilde \nabla $-parallel sections. Let $\mu \colon
\Lambda \rightarrow \automorphism({\mathfrak g}_0)$ denote the monodromy
representation of the flat connection $\nabla $. Then this
representation is by Lie algebra automorphisms and, by Lie II for Lie
groups, lifts to a Lie group homomorphism $\nu \colon \Lambda
\rightarrow \automorphism(G_0)$, where $G_0$ is the simply-connected
Lie group integrating ${\mathfrak g}_0$. The twisted Lie algebra
action $({\mathfrak g},\nabla)$ is integrated by the twisted Lie group
action $G_0 \times_\nu M$.

By {\df complete} we mean that $\nabla$ should have complete
geodesics. A {\df geodesic} of $\nabla $ is a smooth ${\mathfrak
  g}$-path $t \mapsto X_t \in {\mathfrak g}$ such that
$\nabla_{\#X_t}X_t=0$, where $\# \colon {\mathfrak g} \rightarrow TM$
denotes the anchor. A {\df ${\mathfrak g} $-path} is a path $t \mapsto
X_t \in {\mathfrak g}$ such that $\#X_t=\dot m_t$, where $m_t \in M$
is the base point of $X_t$ and a dot denotes derivative.

Compactness of $M$ is not sufficient for completeness of $\nabla$
unless $M$ is simply-connected. A counterexample, and some sufficient
conditions for completeness, are offered in \cite{Blaom_13}.

\subsection{Finite type Lie pseudogroups}\lab{goofy}
Let $J^k(M\times M)$ denote the Lie groupoid of $k$-jets of
diffeomorphisms $\varphi \colon U \rightarrow V$ from an open set $U
\subset M$ to an open set $V \subset M$. Let $\Gamma$ denote a
pseudogroup of infinitely differentiable transformations on
$M$. Prolonging, one obtains a tower of surjective maps between
groupoids,
\[
\Gamma^0 \leftarrow \Gamma^1 \leftarrow
\Gamma^2 \leftarrow \cdots.
\]
Here $\Gamma^i := \{ J^i_m \varphi\suchthat \varphi \in \Gamma,\, m
\in \domain(\varphi)\}$ and $\Gamma^0$ is the foliation by orbits of the
pseudogroup, understood as a subgroupoid of $M \times M$. We will call
$\Gamma$ a {\df Lie pseudogroup of finite type $k$} if the
subgroupoids $\Gamma^k \subset J^k(M \times M)$ and $\Gamma^{k+1}
\subset J^{k+1}(M \times M)$ are {\em Lie} subgroupoids, and if the
surjection $\Gamma^{k+1} \rightarrow \Gamma^{k}$ is a diffeomorphism.

Finite type Lie pseudogroups commonly arise as the symmetries of
sufficiently regular finite type geometric structures.

\begin{theorem}
  For every Lie pseudogroup $\Gamma $ of finite type $k$, there exists
  a canonical pseudoaction ${\mathcal F} $ on $\Gamma^k$ such that
  $\Gamma = \pseud({\mathcal F})$. 
\end{theorem}
\noindent%
Applying   Theorem
\ref{lie3} we obtain:
\begin{corollary}
  If the Lie groupoid $\Gamma^k$ is source-connected, then the Lie
  pseudogroup $\Gamma $ defining it is generated by infinitesimal data
  --- i.e., by the flows determined by a flat Cartan connection $\nabla $
  on the Lie algebroid of $\Gamma^k$, as described in \S\ref{hsu}.
\end{corollary}

The construction of ${\mathcal F} $ is as follows: Regarding
$\Gamma^{k+1}\subset J^{k+1}(M \times M)$ as a subgroupoid of
$J^1(J^k(M \times M))$, we in fact have $\Gamma^{k+1} \subset J^1
\Gamma^k$. The smooth inverse $S \colon \Gamma^k \rightarrow
\Gamma^{k+1} \subset J^1 \Gamma^k$ of $\Gamma^{k+1}\rightarrow
\Gamma^{k}$ (well-defined by $S(J^k_m \varphi)=J^{k+1}_m \varphi$)
becomes a right-inverse for the natural projection
$J^1\Gamma^k\rightarrow \Gamma^k$ and is necessarily a groupoid
morphism. In other words, $S$ is a Cartan connection on $\Gamma^k$.

The corresponding distribution $D$ on $\Gamma^k$ is
integrable. Indeed, each point in $\Gamma^k$ is of the form
$g=J_{m_0}^k\varphi$ for some $\varphi \in \Gamma$ and $m_0 \in M$; a
local bisection of $\Gamma^k$ whose image integrates $D$ and contains
$g$ is given by $m \mapsto J^k_m\varphi$. The integrating foliation
${\mathcal F}$ is a pseudoaction, by Proposition \ref{onepoint}.

The proof that $\Gamma = \pseud({\mathcal F})$
is postponed to \S\ref{lp}

\vspace{0.75\baselineskip}%
While every Lie pseudogroup of finite type is generated by a
pseudoaction, not all pseudogroups generated by a pseudoaction are Lie
pseudogroups. For example, the pseudogroup of local flows of a vector
field on ${\mathbb R} $, vanishing at zero to all orders of
differentiability, but vanishing at no other point, is not a Lie
pseudogroup of finite type. However, Theorem \ref{lie3} guarantees
that every pseudogroup generated by the flow of a vector field is
generated by a pseudoaction: take ${\mathfrak g}$ to be the action
algebroid ${\mathbb R} \times M$.

\subsection{Non-associative local Lie groups}
We now describe what is perhaps the simplest non-trivial example of a
pseudoaction on a simply-connected domain (where there can be no
monodromy twist). The example is interesting because the underlying
Lie algebra is commutative but the corresponding pseudogroup is not.

Let ${\mathbb R}^2\backslash 0$ denote the punctured plane and $\pi
\colon M \rightarrow {\mathbb R}^2\backslash 0 $ its universal
covering. The coordinate vector fields $\partial / \partial x$ and
$\partial/ \partial y$ on ${\mathbb R}^2 $ pull back to commuting
vector fields $v_1, v_2$ on $M$ which are consequently the
infinitesimal generators of an action of the Lie algebra ${\mathbb
  R}^2 $ on $M$. This action is evidently not complete. If $\pi(m) =
(-1/\sqrt{2},-1/\sqrt{2})$ then $\exp(v_1)(\exp(v_2) m)\ne
\exp(v_2)(\exp(v_1) m)$, which shows the pseudogroup generated by this
Lie algebra action is not commutative. Consequently, while $M$ is
locally isomorphic to the the Abelian group ${\mathbb R}^2 $, there
can be no way to extend a local isomorphism to a global one. Olver
\cite{Olver_96} shows that one can give $M$ the structure of a local
Lie group, but multiplication in this structure is not associative.

The action algebroid in this example is isomorphic to $TM$ which is the
Lie algebroid of the pair groupoid $M \times M$. Now
\begin{equation*}
  \Omega(m_1,m_2) = \pi(m_2) - \pi(m_1)
\end{equation*}
defines a Lie groupoid morphism $\Omega \colon M \times M \rightarrow
{\mathbb R}^2 $ into the Abelian Lie group ${\mathbb R}^2 $ and the
derivative of $\Omega $ is evidently the Lie algebroid morphism
$\omega \colon TM \rightarrow {\mathbb R}^2$ satisfying
$\omega(v_1(m))=(1,0)$ and $\omega (v_2(m))=(0,1)$. It is easy to see
that the connected components of fibres of $\Omega \colon M \times M
\rightarrow {\mathbb R}^2 $ must be the leaves of a pseudoaction
${\mathcal F} $ on $M \times M$ and that this pseudoaction formally
integrates --- and hence integrates --- the Lie algebra action.

The lack of globalizability noted above is reflected in interesting
topology in the corresponding pseudotransformations (the leaves of
${\mathcal F} $). To describe this topology, first note that the
target map $\alpha (m_1,m_2)=m_1$ maps $\Omega^{-1}(u)$ onto
$M_u:=M\backslash\pi^{-1}(-u)$, $u \in {\mathbb R}^2$. If $u=0$ then
$M_u=M\cong {\mathbb R}^2 $; if $u \ne 0$ then $M_u$ is a plane with a
countably infinite number of discrete punctures.

Now identify the group of covering transformations of $\pi \colon M
\rightarrow {\mathbb R}^2\backslash 0$ with ${\mathbb Z} $ and let
this group act on $M \times M$ by acting on the second factor. Then
${\mathbb Z} $ acts freely and transitively on the fibres of the
restriction $\alpha \colon \Omega^{-1}(u) \rightarrow M_u$, so that
$\Omega^{-1}(u)$ is a principal ${\mathbb Z} $-bundle over $M_u$. If
$u=0$ this bundle is trivial, $\Omega^{-1}(0) \cong {\mathbb Z} \times
M$, and the connected components of $\Omega^{-1}(0)$ may be viewed as
bona fide global transformations of $M$, namely the covering
transformations of $\pi \colon M \rightarrow {\mathbb R}^2 \backslash
0$. However, for $u \ne 0$ the principal ${\mathbb Z}$-bundle
$\Omega^{-1}(u) $ is non-trivial (one can show the associated \^Cech
cohomology class $c \in H^1(M_u, \underline{{\mathbb Z}})$ is
non-trivial). It is, however, connected, and hence a leaf of
${\mathcal F} $. From these observations it follows that the leaf
space of ${\mathcal F} $ is the plane with a countable infinity of
origins.

\vspace{0.75\baselineskip} The example above was pointed out to us by
a referee who queries the extent to which our results depend upon
unwitting assumptions of associativity. The only associativity used in
this paper is associativity of multiplication in a (global) Lie group
and associativity of map composition (pseudogroups are always
associative).

\section{Elementary considerations}
\subsection{Elements of $J^1 G$ and its Lie algebroid}\lab{jet}
Let $J^1 G$ denote the Lie groupoid of one-jets of local bisections of
a Lie groupoid $G$. Given an arrow $g \in G$ from $m $ to $m'$, we
will view an element of $J^1_gG$ as a linear map $\mu \colon {T_mM}
\rightarrow T_gG$ such that $T_m \alpha \circ \mu =
\identity_{{T_mM}}$ and $T_m \beta \circ \mu\colon T_mM \rightarrow
T_{m'}M$ is invertible. Here $\alpha $ and $\beta $ denote the source
and target maps respectively.

If ${\mathfrak g}$ is the Lie algebroid of $G$ then by convention
elements of ${\mathfrak g} $ will be regarded as vectors tangent at
some $m \in M \subset G$ to a fibre of the source projection $\alpha
\colon G \rightarrow M$. Let $J^1 {\mathfrak g} $ denote the vector
bundle of one-jets of sections of ${\mathfrak g} $. Implicit in
\S\ref{onepoint} is an identification of $J^1 {\mathfrak g} $ with the
(abstract) Lie algebroid of $J^1 G$. This identification is given by
\begin{equation*}
  J^1_m X \mapsto \frac{d}{dt}T_m(\Phi_{X^\mathrm{R}}^t\circ
  \iota_M)\Big|_{t=0}; \qquad X \in \Gamma({\mathfrak g}),\, m \in M.
\end{equation*}
Here $X^\mathrm{R}$ denotes the right-invariant vector field on $G$
corresponding to $X$, $\Phi_{X^\mathrm{R}}^t$ its time-$t$ flow map,
and $\iota_M \colon M \rightarrow G$ the inclusion.

\subsection{Local bisections integrating a foliation}\lab{terminology}
It will be convenient henceforth to view a local bisection of a Lie
groupoid $G$ as a right inverse $b \colon U \rightarrow G$ for the
target projection $\beta \colon G \rightarrow M$. We say $b$ {\df
  integrates} a foliation ${\mathcal F}$ on $G$ if the image of any
connected subset of $U$ lies within a single leaf of ${\mathcal
  F}$. Assuming the leaves of ${\mathcal F}$ are
pseudotransformations, the set $\widehat {\mathcal F} $ defined in
\S\ref{cause} then consists of all local bisections of $G$ integrating
${\mathcal F} $, as it is not hard to show. We record the following
elementary observations, whose proofs are left to the reader:
\begin{proposition}Let ${\mathcal F}$ be a foliation on $G$ whose
  leaves are pseudotransformations. Then:
  \begin{conditions}
  \item\lab{u1} Any local bisection $b \colon U \rightarrow G$
    integrates ${\mathcal F}$ if and only if the image of the tangent
    map $Tb \colon TU \rightarrow TG$ is contained in the distribution
    $D \subset TG$ tangent to ${\mathcal F}$.
  \item\lab{uniqueness} Any two local bisections integrating
    ${\mathcal F}$, and agreeing at a single point $m \in M$ in their
    common domain $U$, agree on all of $U^m$, where $U^m$ is the
    connected component of $U$ containing $m$.
  \item If ${\mathcal F} $ is a pseudoaction, then every element of
    $\pseud({\mathcal F})$ is, locally, of the form
    $U\xrightarrow{b}G\xrightarrow{\beta}M$, for some local bisection
    $b$ integrating ${\mathcal F}$.
  \end{conditions}
\end{proposition}

\subsection{The proof of Proposition \ref{onepoint}}
Let ${\mathcal F}$ be an arbitrary foliation on a Lie groupoid $G$ and
$D \subset TG$ its tangent distribution. We omit the straightforward
proof that $D$ being Cartan is necessary for ${\mathcal F} $ to be a
pseudoaction, and prove only sufficiency.

Let $S \colon G \rightarrow J^1 G$ denote the Lie groupoid morphism
corresponding to the distribution $D$ tangent to ${\mathcal F} $,
which we suppose is a Cartan connection. So, if $g \in G$ is an arrow
in $G$ beginning at $m \in M$, then $S(g)v \in D(g)$ for all $v \in
T_mM$.  From the definition of a Cartan connection, it is not hard to
see the leaves of ${\mathcal F} $ must be pseudotransformations. To
show ${\mathcal F} $ is multiplicatively closed, let $b_1 \colon U_1
\rightarrow G$ and $b_2 \colon U_2 \rightarrow G $ be two local
bisections integrating ${\mathcal F}$, and suppose they are
composable, i.e., $U_1=\beta(b_2(U_2))$; we need to show that the
product $b_1 b_2 \colon U_2 \rightarrow G$ also integrates ${\mathcal
  F}$.
    
Let $m \in U_2$ be arbitrary. Then, using the definition of products in
$J^1 G$, it is not hard to show that
\begin{equation*}
   T(b_1 b_2) \cdot v= J^1_m(b_1 b_2)v=(J^1_{m'}b_1\,J^1_m b_2)v; \qquad v \in T_m M,\, m':=\beta(b_2(m)).
\end{equation*}
On the other hand, since $b_1$ and $b_2$ integrate ${\mathcal F} $, it
follows from \eqref{u1} above that $J^1_{m'}b_1=S(b_1(m'))$ and $J^1_m
b_2=S(b_1(m))$. Substituting into the above, and using the fact that
$S$ is a Lie groupoid morphism, we conclude 
\[
T(b_1 b_2) \cdot
v=S(b_1(m')b_2(m))v \in D.
\]
Since $m$ and $v$ are arbitrary, the image of $T(b_1 b_2) \colon TU_2
\rightarrow TG$ is contained in $D$.  By \eqref{u1} above, $b_1 b_2$
integrates ${\mathcal F}$, so that $\widehat {\mathcal F} \subset \widehat G$
is closed under multiplication.

The proof that $\widehat {\mathcal F} \subset \widehat {\mathcal G} $ is
closed under inversion is similar and omitted.

\subsection{Proof of Theorem  \ref{goofy}}\lab{lp}
The construction of a pseudoaction ${\mathcal F} $ on $\Gamma^k$, for
any Lie pseudogroup $\Gamma $ of finite type $k$, was given in
\S\ref{goofy}; it remains to show that $\Gamma = \pseud({\mathcal F})
$.  It is easy to see that $\Gamma \subset \pseud({\mathcal F})$. For
the reverse inclusion, suppose $\phi \in \pseud({\mathcal F})$. To
show $\phi \in \Gamma $ it suffices, by the collating property of
pseudogroups, to construct, for any $m_0$ in the domain $U $ of
$\phi$, an open neighbourhood $V$ of $m_0$ such that $\phi|_V \in
\Gamma$.

Since $\phi \in \pseud({\mathcal F})$ we have, shrinking $U \ni m_0$
if necessary, $\phi = \beta \circ b$, for some local bisection $b
\colon U \rightarrow \Gamma^k$ integrating ${\mathcal F}$. Here $\beta
\colon \Gamma^k \rightarrow M$ denotes the target map of
$\Gamma^k$. Furthermore, we have $b(m_0)=J^k_{m_0} \varphi$ for some
$\varphi \in \Gamma$, with domain $U'\ni m_0$, say. But in that case
we obtain a second local bisection $b' \colon U' \rightarrow \Gamma^k$
integrating ${\mathcal F}$, defined by $b'(m)=J^k_m \varphi$, and
satisfying $b'(m_0)=b(m_0)$. By the local uniqueness of integrating
bisections (Proposition \eqrefs{terminology}{uniqueness}) $b$ and $b'$
coincide on some open neighbourhood $V$ of $m_0$, giving us
$\phi|_V=\beta \circ b'|_V=\varphi|_V \in \pseud(\Gamma)$.

\section{Integrating a Lie algebra action}\lab{av}
In this central section of the paper we show that every Lie algebra
action is integrated by a pseudoaction ${\mathcal F}$ which is also a
formal integration. We will construct our integration with the help of
Palais' local integrability theorem, of which we offer a novel
proof. While in principle the reader may take Palais' result as given
--- and read the remainder of the section independently of our proof
of it --- we feel our proof puts the later constructions into better
context.

Let ${\mathfrak g}_0 $ be a finite-dimensional Lie algebra acting
smoothly from the left on $M$, and denote the corresponding Lie
algebra homomorphism ${\mathfrak g}_0 \rightarrow \Gamma(TM)$ by $\xi
\mapsto \xi^\dagger$. The simply-connected Lie group with Lie algebra
${\mathfrak g}_0 $ will be denoted by $G_0$.

\subsection{Local actions}\lab{loc}%
By a {\df local action} of $G_0$ on $M$ we shall mean an open
neighbourhood $W \subset G_0 \times M$ of the identity section
$\identity \times M$, together with a smooth map $(g,m) \mapsto \phi_g
(m)\colon W \rightarrow M$ such that:
\begin{conditions}
  \item\lab{ip1} $\phi_\identity(m) = m $ for all $m \in M$.
  \item\lab{ip3}  $\phi_h(\phi_g(m)) = \phi_{hg}(m)$ whenever both sides are
  well-defined.
\end{conditions}

% \begin{remark}
%   If only \eqref{ip2} does not hold then one obtains a local action by
%   replacing $W$ with $W \cap W^{-1}$, where $W^{-1}=\{(g^{-1}, g \cdot
%   m)\suchthat (g,m) \in W\}$.
% \end{remark}
\noindent% 
Without loss of generality, we can suppose any local action
additionally satisfies:
\begin{conditions}
% \item\lab{ip5} $(\exp(\xi),m ) \in W \Rightarrow (\exp(t \xi), m) \in
%   W$ for all $t \in [0,1]$.
\item\lab{ip4} $(g, m) \in W \Leftrightarrow (g^{-1}, \phi_g(m)) \in
  W$.
\end{conditions}
For if not, we can replace $W$ with $W \cap W^{-1}$, where $W^{-1}=\{(g^{-1},
\phi_g(m)\suchthat (g,m) \in W\}$.  
% Define a a `ball of radius $r$' by $B_r:=\exp(rW)\subset B$, where
% $0<r\le 1$. Then for each $m \in M$ there exists $r_m>0$ and a
% neighbourhood $U_m$ of $m$ in $M$ such that $B_{r_m} \times U_m
% \subset W$.

\begin{theorem}[Palais' local integrability theorem \cite{Palais_57a}]
  There exists a local action $(m, g) \mapsto \phi_g$ of $G_0$ on $M$
  such that $\phi_g$ is the restriction to some open set of the
  time-one flow map of some infinitesimal generator $\xi^\dagger$,
  $\xi \in {\mathfrak g}_0$.
\end{theorem}
Before turning to the proof, we record a simple lemma that will be
needed both in the proof and later on.
\begin{lemma}
  Let $G$ be a Lie groupoid whose Lie algebroid is the action
  algebroid ${\mathfrak g}_0 \times M$. For each $\xi \in {\mathfrak
    g}_0$ let $\xi_\mathrm{c}$ denote the constant section of
  ${\mathfrak g}_0 \times M$ and $\xi_\mathrm{c}^\mathrm{R}$ the
  corresponding right-invariant vector field on $G$. Then the flow of
  the vector fields $\xi^\dagger$ and $\xi_\mathrm{c}^\mathrm{R}$ are
  related by
  \begin{equation}
    \beta(\Phi^t_{\xi_\mathrm{c}^\mathrm{R}}(p))=\Phi^t_{\xi^\dagger}(\beta(p));\qquad p \in G,\mathlab{bsj1}
  \end{equation} whenever both sides are defined. Here $\beta $
  denotes the target projection of $G$. Also, if
  $\Phi^t_{\xi^\dagger}(\beta(p))$ is defined for some $t \in {\mathbb
    R} $, then so is $\Phi^t_{\xi_\mathrm{c}^\mathrm{R}}(p)$,  i.e.,
  \begin{equation} \domain \Phi^t_{\xi_\mathrm{c}^\mathrm{R}}=
    \beta^{-1}\left(\domain \Phi^t_{\xi ^\dagger}\right).\mathlab{bsj2}
  \end{equation} 
  Finally, if $\Omega \colon G \rightarrow G_0$ is a Lie groupoid
  morphism whose derivative is the canonical projection ${\mathfrak
    g}_0 \times M \rightarrow {\mathfrak g}_0$, then 
  \begin{equation} 
    \Omega(\Phi^t_{\xi_\mathrm{c}^{\mathrm
        R}}(p))=\exp(t \xi)\Omega(p); \qquad \xi \in {\mathfrak
      g}_0,\enspace p \in G.\mathlab{identity}
  \end{equation}
\end{lemma}
\begin{proof}
  Equation \eqref{bsj1} is an immediate consequence of the fact that
  $\xi_\mathrm{c}^\mathrm{R}$ and $\xi^\dagger$ are $\beta
  $-related. To see that \eqref{bsj2} holds note that integral paths
  $t \mapsto m(t)$ of $\xi^\dagger=\# \xi_\mathrm{c}$ lift to
  ${\mathfrak g}$-paths $t \mapsto (\xi,m(t))$ (${\mathfrak
    g}={\mathfrak g}_0 \times M$), which in turn integrate to families
  of $G$-paths (paths lying in source-fibres).\footnote{To see that
    ${\mathfrak g}$-paths lift to $G$-paths the essential points are:
    (i) The base-paths of ${\mathfrak g}$-paths always lie on orbits
    of $G$; and (ii) The restriction of $\beta \colon G \rightarrow M$
    to any source-fibre of $G$ is a principal bundle over the
    corresponding orbit. See, e.g., \cite{Crainic_Fernandes_11}. }
  These $G$-paths are necessarily integral paths of
  $\xi_\mathrm{c}^\mathrm{R}$ mapped by $\beta $ to integral paths of
  $\xi^\dagger$.

  To prove \eqref{identity}, let $\xi^\mathrm{R}$ denote the
  right-invariant vector field on the Lie group $G_0$ corresponding to
  $\xi \in {\mathfrak g}_0$. Then $\xi^\mathrm{R}_\mathrm{c}$ and
  $\xi^\mathrm{R}$ will be $\Omega$-related. It follows that $\Omega $
  maps integral paths of $\xi^\mathrm{R}_\mathrm{c}$ to integral paths
  of $\xi^\mathrm{R}$, and so \eqref{identity} holds.
\end{proof}

\subsection{Proof of Palais' local integrability theorem}\lab{hgn}
According to Dazord, all action algebroids are integrable. That is to
say, there exists a Lie groupoid $G$ whose Lie algebroid is
${\mathfrak g}_0 \times M$. We refer to Dazord's original article
\cite{Dazord_97} for the elegant proof, but will make no use of the
explicit model for $G$ constructed there. (An alternative model
appears in \S\ref{concrete} below.) By Lie I for Lie groupoids --- see
e.g., \cite{Crainic_Fernandes_11} --- we may take $G $ to be source
simply-connected; then using Lie II for Lie groupoids, we obtain a Lie
groupoid morphism $\Omega \colon G \rightarrow G_0$ whose derivative
is the canonical projection $\omega \colon {\mathfrak g}_0 \times M
\rightarrow {\mathfrak g}_0$.

We next construct `tubular neighbourhood' of $M$ in $G$ modelled on a
neighbourhood of $\identity \times M$ in $G_0 \times M$. Let $B
\subset G_0$ be any open neighbourhood of the identity that is the
diffeomorphic image of some convex open neighbourhood of zero in
${\mathfrak g}_0$, under the exponential map $\exp \colon {\mathfrak
  g}_0 \rightarrow G_0$.  Then, by local existence and uniqueness
theorems for ODE's, there exists an open neighbourhood $W_\mathrm{big}
\subset B \times M$ of $\identity \times M$ such that $E \colon
W_\mathrm{big}\rightarrow G$, given by
$E(\exp(\xi),m)=\Phi^1_{\xi^\mathrm{R}_\mathrm{c}} (m)$, is
well-defined (notation is as in Lemma \ref{loc}).

Evidently, the tangent map $TE \colon T{W_\mathrm{big}}\rightarrow TG$
has full rank at all points of $\identity \times M$, implying that $E$
is a local diffeomorphism in some neighbourhood of $\identity \times
M$. By a topological argument familiar from proofs of the tubular
neighbourhood theorem (see, e.g., \cite[IV, \S5, p.~109]{Lang_95}) the
reader will agree that $E$ is in fact a global diffeomorphism, when
restricted to a possibly smaller neighbourhood of $\identity \times M$
in $B\times M$, which we nevertheless continue to denote by
$W_\mathrm{big}$.

Notice that by construction $\alpha(E(g,m))=m$, where $\alpha $
denotes the source map, and that 
\begin{equation}
  \Omega(E(g,m))=g, \mathlab{gor}
\end{equation}
by Equation \eqref{bsj1} of the preceding lemma. Consequently, $\Omega
\times \alpha \colon E(W_\mathrm{big}) \rightarrow W_\mathrm{big}$ is
the inverse of $E$; in particular $\Omega \times \alpha $ is injective
on $Z_\mathrm{big}:=E(W_\mathrm{big})$. Note that $Z_\mathrm{big}$ is
paracompact and Hausdorff, because $W_\mathrm{big}\subset B \times M$
is evidently paracompact and Hausdorff.

We now apply the following fact proven in Appendix \ref{gpmu}, which
generalizes a well-known observation about Lie groups:
\begin{proposition}
  Let $Z_\mathrm{big} \subset G$ be a paracompact, Hausdorff, open
  neighbourhood of $M$ in a Lie groupoid $G$ over $M$. Then there
  exists a neighbourhood $Z \subset Z_\mathrm{big}$ of $M$ such that
  for all $h,g \in Z$ we have $hg \in Z_\mathrm{big}$, whenever $h$
  and $g $ are multipliable.
\end{proposition}
\noindent%
Setting $W = E^{-1}(Z)$, we seek to show that $\phi_g(m) :=
\beta(E(g,m))$ defines a local action $(g, m) \mapsto \phi_g(m) \colon
W \rightarrow M$ of $G_0$ on $M$; recall $\beta$ denotes the target
map.

The requirement \eqrefs{loc}{ip1} follows immediately. To establish
\eqrefs{loc}{ip3}, suppose $(m,g) \in W$, $(h, \phi_g(m)) \in W$ and
$(hg,m) \in W$. Then $E(h,\phi_g(m))$ and $E(g,m)$ are multipliable
elements of $Z \subset G$. Moreover, we claim
\begin{equation}
  E(h,\phi_g(m))E(g,m)=E(hg,m).\mathlab{trt}
\end{equation}
As $Z Z \subset Z_\mathrm{big}$ (in the sense of the proposition) both sides
of the equation are elements of $Z_\mathrm{big}$, on which $\Omega
\times \alpha $ is injective. To show equality holds in \eqref{trt} it
therefore suffices to show equality after applying $\Omega \times
\alpha $ to both sides of the equation. But this follows easily from
\eqref{gor} and the fact that $\Omega \colon G \rightarrow G_0$ is a
groupoid morphism.

We now compute
\begin{multline*}
  \phi_h(\phi_g(m)))=\beta\Big(\,E(h,\phi_g(m))\,\Big) =\beta
  \Big(\,E(h,\phi_g(m))E(g,m)\,\Big)\\=\beta
  \Big(\,E(hg,m)\,\Big)=\phi_{hg}(m).
\end{multline*}
The third equality follows from \eqref{trt}. This completes the proof
that $(g,m) \mapsto \phi_g (m) \colon W \rightarrow M$ is a local
action.

Note that the definition of $E$ and \eqrefs{loc}{bsj1} imply that
$\phi_{\exp(\xi)}=\Phi^1_{\xi^\dagger}$, so that the local
transformations $\phi_g$ have the form asserted in the theorem.

\subsection{An explicit Lie groupoid integrating ${\mathfrak g}_0
  \times  M$}
\lab{concrete}%
Fix a local action $(m,g) \mapsto \phi_g(m)\colon W \rightarrow M$ of
$G_0$ on $M$ integrating the action of ${\mathfrak g}_0 $ on $M$, in
the sense of Theorem \ref{loc}. In the sequel a statement such as `$m
\in \domain \phi_g$' should be interpreted as `$(g,m) \in W$'. We
assume \eqrefs{loc}{ip4} holds, i.e., $m \in \domain\phi_g
\Leftrightarrow \phi_g(m) \in \domain \phi_{g^{-1}}$. Evidently,
$\phi_{g^{-1}}=\phi_g^{-1}$.

By a {\df chain} for the local action, let us mean a finite sequence
$(g_1,g_2,\ldots,g_k,m)$, where $g_1, g_2, \ldots, g_k \in G_0$ and $m
\in M$ are such that $m \in \domain \phi_{g_1}\circ \phi_{g_2}\circ
\cdots \circ \phi_{g_k}$. Explicitly, this means:
\begin{conditions}
\item $m_k:=m \in \domain \phi_{g_k}$,
\item $m_j:=\phi_{g_{j+1}}(m_{j+1})\in \domain \phi_{g_{j}} $, for $j
  =k-1, k-2, \ldots, 1$.
\end{conditions}
An equivalence relation $\sim$ on chains is defined by declaring
\[(g_1,g_2,\ldots,g_k,m)\sim (\bar g_1,\bar g_2,\ldots,\bar g_{\bar
  k},\bar m)\] if:
\begin{conditions}
  \item $m={\bar m}$,
  \item $g_1 g_2 \ldots g_k = \bar g_1\bar g_2 \ldots \bar g_{\bar
      k}$,
  \item $\germ_m \phi_{g_1}\circ \phi_{g_2}\circ \cdots \circ
    \phi_{g_k}=\germ_{{\bar m}}\phi_{\bar g_1}\circ \phi_{\bar
      g_2}\circ \cdots \circ \phi_{\bar g_{\bar k}}$.
\end{conditions}

We define $G$ to be the set of all chains modulo the equivalence
relation $\sim$. The class represented by a chain
$(g_1,g_2,\ldots,g_k,m)$ will be denoted $[g_1,g_2,\ldots,g_k,m]$.
\begin{theorem}
  The set $G$ is a Lie groupoid over $M$ with Lie
  algebroid ${\mathfrak g}_0 \times M$. Furthermore, the map $\Omega
  \colon G \rightarrow G_0$ defined
  by
  \[
  \Omega([g_1,g_2,\ldots,g_k,m])=g_1 g_2 \cdots g_k
  \]
  is a Lie groupoid morphism whose derivative is $\omega \colon
  {\mathfrak g}_0 \times M \rightarrow {\mathfrak g}_0 $.
\end{theorem}
That $G$ is a set-theoretic groupoid over $M$ is clear. The groupoid
operations closely resemble those for action groupoids: The identity
at $m \in M $ is $[\identity, m]$. The source and targe maps
$\alpha,\beta:G \rightarrow M$ are given by
$\alpha([g_1,g_2,\ldots,g_k,m])=m$, $\beta([g_1,g_2,\ldots,g_k,m])=
(\phi_{g_1}\circ \phi_{g_2}\circ \cdots \circ
\phi_{g_k})(m)$. Multiplication is defined by
\[
[g_1,g_2,\ldots,g_k,m][\bar g_1,\bar g_2,\ldots,\bar g_{\bar
  k},\bar m]=[g_1,g_2,\ldots,g_k,\bar g_1,\bar g_2,\ldots,\bar g_{\bar
  k},\bar m].
\]
The inverse of $[g_1,g_2,\ldots,g_k,m]$ is
$[g_k^{-1},g_{k-1}^{-1},\ldots,g_1^{-1},(\phi_{g_1}\circ
\phi_{g_2}\circ \cdots \circ \phi_{g_k})(m)]$, which is well-defined
on account of the requirement \eqrefs{loc}{ip4}.

After defining the smooth structure of $G$ in \S\ref{smooth} below, it
is a straightforward exercise left to the reader to show this
structure is compatible with the groupoid operations, and that the
following map determines an isomorphism between the action algebroid
${\mathfrak g}_0 \times M$ and the abstract Lie algebroid of $G$:
\begin{equation*}
  (\xi, m)\mapsto \frac{d}{dt}\big[\,\exp(t \xi), m\,\big]\Big|_{t=0}.
\end{equation*}
%
% To see that $G$ is source-connected, observe that for any arbitrary
% point $[g_1,g_2,\ldots g_k,m]\in G$ one may define a path $t \mapsto
% p_t \colon [0,1] \rightarrow \alpha^{-1}(m)$ joining
% $[g_2,\ldots,g_k,m]$ to $[g_1,g_2,\ldots g_k,m]$ by $p_t=[\exp(t
% \xi),g_2,\ldots g_k,m]$, where $\xi = \exp^{-1}(g_1)$. With respect to
% the smooth structure defined in \S\ref{smooth}, one readily shows this
% path is continuous (indeed smooth). An obvious inductive argument
% shows that $[g_1,g_2,\ldots g_k,m]\in G$ and $[\identity,m]$ lie in
% the same path-component of $\alpha^{-1}(m)$.
The claim regarding $\Omega $ is immediately verified.

\subsection{The smooth structure of $G$}\lab{smooth} Define $B \subset
G_0$ as in \S\ref{hgn} and define a `ball of radius $r$' by
$B_r:=\exp(r\exp^{-1}(B))\subset B$; here $0<r\le 1$. To construct a
smooth atlas on $M$ will require the following preliminary
observations:
\begin{lemmaA}\mbox{}
  \begin{conditions}
    \item \lab{local1} For every relatively compact open set $V\subset M$ there
      exists $r>0$ such that $V \subset \domain \phi_g$ for all $g\in B_r$.
    \item\lab{local2} Given $r>0$ and $g \in B_r$, there exists $\rho >0$ such
      that $B_\rho g \subset B_r$.
    % \item\lab{local3} Let $g,h$ be elements of $B$ and $V \subset M$
    %   an open connected set. Suppose $hg \in B $, $V \subset \domain
    %   \phi_g$, $\phi_g(V)\subset \domain \phi_h$, and $V \subset
    %   \domain \phi_{hg}$. Then $\phi_h \circ \phi_g|_V =
    %   \phi_{hg}|_V$.
  \end{conditions}
\end{lemmaA}
\begin{proof}
  As $V$ is relatively compact, one can readily construct a
  neighbourhood of $V$ in $W$ of the form $B_r \times V$, for some
  $r>0$. The claim \eqref{local1} immediately follows. Conclusion
  \eqref{local2} follows from the continuity of multiplication in
  $G_0$.
\end{proof}

For each $[g_1,g_2,\ldots,g_k,m_0] \in G$ we define a local
parameterization $\varphi$ of $G$ as follows. First, let $U$ be a
relatively compact open neighbourhood of $m_0$ contained in the domain
of $\phi_{g_1}\circ \phi_{g_2}\circ \cdots \circ \phi_{g_k}$ and put
\[
V:= (\phi_{g_1}\circ \phi_{g_2}\circ \cdots
\circ \phi_{g_k})(U).
\]
Applying \eqref{local1}, choose $r>0$ small enough that $V \subset
\domain \phi_g$ for all $g \in B_r$. Then $\varphi \colon B_r \times U
\rightarrow G $ is well-defined by
\begin{equation*}
  \varphi(g,m) = [g,{g_1},{g_2},\ldots,g_{g_k},m].
\end{equation*}
That $\varphi$ is injective is a triviality.
  
For the purpose of examining the transition functions associated with
the collection of all such local parameterizations, consider a second
point $[\bar g_1,\bar g_2,\ldots,\bar g_{\bar k},\bar m_0] \in G$, and
an associated local parameterization $\bar\varphi \colon B_{\bar r}
\times \bar U \rightarrow G$, defined analogously. Furthermore,
suppose that the images of the parameterizations have non-trivial
overlap
\[ {\mathcal O} :=\varphi(B_r \times U)\cap \bar\varphi(B_{\bar r}
\times \bar U)\subset G.\] Then, for some $m' \in U \cap \bar U$,
there exists $g_0 \in B_r$ and $\bar g_0 \in B_{\bar r}$ such that
$\varphi(g_0,m')=\bar\varphi(\bar g_0,m')$. In particular, we will
have
\begin{equation}
  g_0 g_1 g_2 \cdots g_k = \bar g_0 \bar g_1 \bar g_2 \cdots \bar
  g_{\bar k}.\mathlab{ee0}
\end{equation}
Given this, one readily shows that the transition function
$\bar\varphi^{-1}\circ \varphi\colon \varphi^{-1}({\mathcal
  O})\rightarrow \bar\varphi^{-1}({\mathcal O})$ is given by $(g,m)
\mapsto (g g_0^{-1}\bar g_0, m)$ which, as a map on $B_r \times U$, is
smooth. To show the local parameterizations define a smooth structure
on $G$, it remains only to establish the following fact, whose proof
will rest crucially on the fundamental property \eqrefs{loc}{ip3} of
the local action:
\begin{lemmaB}
  The set $\varphi^{-1}({\mathcal O})$ is open in $B_r \times U$.
\end{lemmaB}
\begin{proof}
  An arbitrary point of $\varphi^{-1}({\mathcal O})$ is of the form
  $(g_0, m')$ for some $g_0 \in B_r$ and $m' \in U \cap \bar U$
  satisfying $\varphi(g_0,m')=\bar\varphi(\bar g_0,m')$, for some
  $\bar g_0 \in B_{\bar r}$. This means \eqref{ee0} holds and
\begin{equation*}
  \germ_{{m'}} \phi_{g_0}\circ \phi_{g_1}\circ \phi_{g_2}\circ \cdots \circ
  \phi_{g_k}=\germ_{{{m'}}}\phi_{\bar g_0}\circ \phi_{\bar g_1}\circ \phi_{\bar
    g_2}\circ \cdots \circ \phi_{\bar g_{\bar k}}.
\end{equation*}
In particular, there exists an open neighbourhood $U' \subset U \cap
\bar U  $ of ${m'}$ such that
\begin{equation}
  \phi_{g_0}\circ \phi_{g_1}\circ \phi_{g_2}\circ \cdots \circ
  \phi_{g_k}\big|_{U'}=\phi_{\bar g_0}\circ \phi_{\bar g_1}\circ \phi_{\bar
    g_2}\circ \cdots \circ \phi_{\bar g_{\bar
      k}}\big|_{U'}\enspace.\mathlab{ee9}  
\end{equation}
To prove the lemma, we will show that the neighbourhood $B_\rho g_0
\times U'$ of $(g_0, {m'})$ lies in $\varphi^{-1}({\mathcal O})$, for
some some $\rho > 0$. To this end, apply Lemma A to find $\rho > 0 $
small enough that all the following hold, for all $\delta \in B_\rho$:
\begin{gather}
  B_\rho g_0 \subset B_r, \mathlab{ee1}\\
  B_\rho \bar g_0 \subset B_{\bar r}, \mathlab{ee2}\\
  \phi_{g_0}(V) \subset \domain \phi_\delta, \mathlab{ee3}\\
  \phi_{\bar g_0}(\bar V) \subset \domain \phi_\delta. \mathlab{ee4}
\end{gather}
With $\delta \in B_\rho $ arbitrary henceforth, we have, by
\eqref{ee1} and \eqref{ee2},
\begin{gather}
  V \subset \domain \phi_{\delta
    g_0},\mathlab{ee5}\\
  \bar V \subset \domain \phi_{\delta \bar g_0}.\mathlab{ee6}
\end{gather}
% \begin{gather}
%   U \subset \domain \phi_{\delta g_0}\circ \phi_{g_1}\circ
%   \phi_{g_2}\circ \cdots \circ
%   \phi_{g_k}, \\
%   \bar U \subset \domain \phi_{\delta \bar g_0}\circ \phi_{\bar
%     g_1}\circ \phi_{\bar g_2}\circ \cdots \circ \phi_{\bar g_{\bar
%       k}}.
% \end{gather}
% On the other hand, from \eqref{ee3} and
% \eqref{ee4}, we also have
% \begin{gather}
%   U \subset \domain \phi_{\delta}\circ \phi_{g_1}\circ
%   \phi_{g_2}\circ \cdots \circ
%   \phi_{g_k}, \\
%   \bar U \subset \domain \phi_{\delta}\circ \phi_{\bar
%     g_1}\circ \phi_{\bar g_2}\circ \cdots \circ \phi_{\bar g_{\bar
%       k}}.
% \end{gather}
By \eqrefs{loc}{ip3}, we have 
\begin{gather}
  \phi_{\delta g_0}\big|_V = \phi_\delta \circ \phi_{g_0}\big|_V,
  \mathlab{ee7}\\
  \phi_{\delta \bar g_0}\big|_{\bar V} = \phi_\delta \circ \phi_{\bar
    g_0}\big|_{\bar V},  \mathlab{ee8}
\end{gather}
where all maps appearing are well-defined, on account of equations
\eqref{ee3}--\eqref{ee6}.

By \eqref{ee1} and \eqref{ee2}, the open set $B_\rho g_0 \times U'$
lies in $B_r \times U = \domain \varphi$ and the open set $B_\rho \bar
g_0 \times U'$ lies $B_{\bar r} \times \bar U = \domain
\bar\varphi$. Our proof of the lemma is complete if we can show
$\varphi(B_\rho g_0 \times U') \subset \bar\varphi(B_\rho \bar g_0
\times U') $. It suffices to show that $\varphi(\delta g_0, m) =
\bar\varphi(\delta \bar g_0, m)$ for arbitrary $\delta \in B_\rho$ and
$m \in U'$. This is equivalent to showing:
\begin{gather*}
  \delta g_0 g_1 g_2\cdots g_k = \delta \bar g_0 \bar g_1 \bar g_2
  \cdots \bar g_{\bar k} \\
  \text{and}\quad \germ_{m} \phi_{\delta g_0}\circ\phi_{g_1}\circ
  \phi_{g_2}\circ \cdots \circ \phi_{g_k}=\germ_{m}\phi_{\delta \bar
    g_0}\circ \phi_{\bar g_1}\circ \phi_{\bar g_2}\circ \cdots \circ
  \phi_{\bar g_{\bar k}}.
\end{gather*}
The first equation follows immediately from \eqref{ee0}. Regarding the
second, we have
\begin{align*}
  \germ_{m} \phi_{\delta g_0}\circ\phi_{g_1}\circ \phi_{g_2}\circ
  \cdots \circ \phi_{g_k}&=\germ_{m} \phi_{\delta}\circ
  \phi_{g_0}\circ\phi_{g_1}\circ \phi_{g_2}\circ \cdots \circ
  \phi_{g_k}, \qquad \text{by
    \eqref{ee7}}\\
  &=\germ_{m}\phi_{\delta}\circ \phi_{\bar g_0}\circ \phi_{\bar
    g_1}\circ \phi_{\bar g_2}\circ \cdots \circ
  \phi_{\bar g_{\bar k}}, \qquad \text{by \eqref{ee9}}\\
  &=\germ_{m}\phi_{\delta \bar g_0}\circ \phi_{\bar g_1}\circ
  \phi_{\bar g_2}\circ \cdots \circ \phi_{\bar g_{\bar k}},\qquad
  \text{by \eqref{ee8}.}
\end{align*}
\end{proof}

\subsection{Integrability}\lab{tuy}
Let $\nabla $ denote the canonical flat Cartan connection on
${\mathfrak g}_0 \times M$. The following result is a special instance
of Lie III for pseudoactions (Theorem \ref{lie3}):
\begin{proposition}
  There exists a source-connected Lie groupoid $G$ supporting a
  pseudoaction ${\mathcal F} $ that both integrates and formally
  integrates $({\mathfrak g}_0 \times M,\nabla)$.
\end{proposition}
\noindent% 
Specifically, let $G$ denote the Lie groupoid integrating ${\mathfrak
  g}_0 \times M$ defined in \S\ref{concrete} and $\Omega \colon G
\rightarrow G_0$ the Lie groupoid morphism whose derivative is the
canonical projection $\omega \colon {\mathfrak g}_0 \times M
\rightarrow {\mathfrak g}_0 $, and whose explicit form appears in
Theorem \ref{concrete}. Replace $G$ by its source-connected
component\footnote{We believe $G$ is already source-connected ---
  indeed source-simply-connected --- but do not prove or need this
  fact here.} and let ${\mathcal F}$ be the pseudoaction whose leaves
are the connected components of fibres of the submersion $\Omega
\colon G \rightarrow G_0$.
\begin{proof}[Proof that ${\mathcal F}$ integrates ${\mathfrak g}_0
  \times M$]%
  Our task is to prove $\pseud(\nabla) = \pseud({\mathcal F})$. To
  prove $\pseud(\nabla) \subset \pseud({\mathcal F})$, it suffices to
  show that for any $\xi \in {\mathfrak g}_0 $ and $t \in {\mathbb R}
  $, we have $\Phi_{\xi^\dagger}^t \in \pseud({\mathcal F})$. Here and
  below we adopt the notation of Lemma \ref{loc}. By part \eqref{bsj2}
  of that lemma, $U:=\domain \Phi_{\xi^\dagger}^t \subset M \subset G$
  lies in $\domain \Phi_{\xi_\mathrm{c}^\mathrm{R}}^t$ and a local
  bisection $b \colon U \rightarrow G$ is therefore well-defined by
  $b(m):=\Phi^t_{\xi_\mathrm{c}^\mathrm{R}}(m)$. Moreover, by
  \eqrefs{loc}{identity}, we have
  \begin{equation} 
    \Omega(b(m))=\exp(t \xi),\quad\text{for all~}m \in
    U,\mathlab{hair}
  \end{equation} 
  which shows that $b$ integrates ${\mathcal F}$, because leaves of
  ${\mathcal F}$ are connected components of fibres of $\Omega$. It
  follows that $\beta \circ b \in \pseud({\mathcal F})$. But, by
  \eqrefs{loc}{bsj1}, $\beta \circ b =\Phi^t_{\xi^\dagger}$, implying
  $\Phi^t_{\xi^\dagger}\in \pseud({\mathcal F})$, as desired.

  Suppose $\phi = \beta \circ b$ for some local bisection $b \colon U
  \rightarrow G$ integrating ${\mathcal F}$, $U=\domain \phi$. Then as
  all elements of $\pseud({\mathcal F})$ are locally of this form, it
  suffices in establishing $\pseud({\mathcal F}) \subset
  \pseud(\nabla)$ to show an arbitrary point $m \in U$ has an open
  neighbourhood $\tilde U$ such that $\phi|_{\tilde U} \in
  \pseud(\nabla)$. This follows from the collating property of
  pseudogroups.
   
  We have $b(m)=[g_1,g_2,\ldots g_k,m]$ for some $g_1,g_2 \in B$. In
  that case $U'=\domain \phi_{g_1} \circ \phi_{g_2} \circ \cdots \circ
  \phi_{g_k}$ is an open neighbourhood of $m$ and the map $b^m \colon
  U' \rightarrow G$ given by $b^m(m')=[g_1,g_2,\ldots, g_k,m']$ is a
  local bisection integrating ${\mathcal F} $ such that
  $b^m(m)=b(m)$. By the local uniqueness of local bisections
  integrating ${\mathcal F}$ (Proposition
  \eqrefs{terminology}{uniqueness}) we conclude $b|_{\tilde
    U}=b^m|_{\tilde U}$, where $\tilde U$ is the connected component
  of $U \cap U'$ containing $m$. Consequently 
  \[
    \phi|_{\tilde  U}=\phi_{g_1} \circ \phi_{g_2} \circ \cdots \circ
  \phi_{g_k}|_{\tilde U}
  \]
  and whence $\phi|_{\tilde U}\in \pseud(\nabla)$ as claimed.
\end{proof}
\begin{proof}[Proof that ${\mathcal F}$ formally integrates
  ${\mathfrak g}_0 \times M$]
  Recalling the definitions of \S\ref{onepoint} and, in particular,
  the correspondence \eqrefs{onepoint}{cspan}, we see that showing
  ${\mathcal F} $ formally integrates $({\mathfrak g}_0 \times M,
  \nabla)$ amounts to showing $dS(x)=J_m^1X + (\nabla X)(m)$, for any
  $x \in {\mathfrak g}_0 \times M $, where $X$ is an arbitrary
  extension of $x$ to a local section of ${\mathfrak g}_0 \times M
  $. Here $S \colon G \rightarrow J^1 G$ is the Cartan connection on
  $G$ which, as a distribution on $G$, is tangent to the leaves of
  ${\mathcal F} $; indeed, from the explicit form of local bisections
  integrating $\mathcal F$ mentioned in the preceding paragraph, $S$
  is given by
  \begin{equation}
    S([g_1,g_2,\ldots,g_k,m])\dot\gamma(0)=
    \frac{d}{ds}[g_1,g_2,\ldots,g_k,\gamma(s)]\Big|_{s=0},\mathlab{g0}
  \end{equation}
  for any path $s \mapsto \gamma(s) \in M$ with $\dot\gamma(0)\in T_mM$.

  Adopting the notation of Lemma \ref{loc}, we take $x=(\xi, m) $, and
  $X = \xi_\mathrm{c}$, and need to prove
  \begin{equation}
    dS(\xi_\mathrm{c}(m))=J_m^1 \xi_\mathrm{c}.   \mathlab{goot}
  \end{equation}
  First, one readily sees that
  \begin{equation*}
    \xi_\mathrm{c}^\mathrm{R}([g_1,g_2,\ldots,g_k,m])=\frac{d}{dt}[\exp(t
    \xi)g_1,g_2,\ldots,g_k,m]\big|_{t=0}
  \end{equation*}
  for any $[g_1,g_2,\ldots,g_k,m]\in G$. This vector field can be
  explicitly integrated locally; in  particular, 
\begin{equation}
  \Phi_{\xi_\mathrm{c}^\mathrm{R}}^t(m)=[\exp(t \xi),m].\mathlab{g1}
\end{equation}
Next, we claim 
\begin{equation}
  S(\Phi_{\xi_\mathrm{c}^\mathrm{R}}^t(m))=T_m(\Phi_{\xi_\mathrm{c}^\mathrm{R}}^t
  \circ \iota_M),\mathlab{g2}
\end{equation}
where $\iota \colon M \rightarrow G$ is the inclusion. Indeed, for any
path $s \mapsto \gamma(s) \in M$ with $\dot\gamma(0)\in T_m M$,
\begin{align*}
  S\left(\,\Phi_{\xi_\mathrm{c}^\mathrm{R}}^t(m)\,\right)\dot\gamma(0)&=S\left(\,[\exp(t
    \xi),m]\,\right)\dot\gamma(0),\qquad&\text{by \eqref{g1}}\\
  &=\frac{d}{ds}[\exp(t \xi),\gamma(s)]\Big|_{s=0},& \text{by
    \eqref{g0}}\\
  &=\frac{d}{ds}\Phi_{\xi_\mathrm{c}^\mathrm{R}}^t(\gamma(s))\Big|_{s=0},&
  \text{by
    \eqref{g1}}\\
  &=T_m(\Phi_{\xi_\mathrm{c}^\mathrm{R}}^t \circ \iota_M)
  \dot\gamma(0).
\end{align*}
Using \eqref{g2} we compute 
\begin{align*}
  dS(\xi_\mathrm{c}(m))=\frac{d}{dt}S(\Phi_{\xi_\mathrm{c}^\mathrm{R}}^t
  (m))\Big|_{t=0}&=\frac{d}{dt}T_m(\Phi_{\xi_\mathrm{c}^\mathrm{R}}^t(m))\Big|_{t=0}\\  
  &=J^1_m \xi_\mathrm{c}^\mathrm{R},
\end{align*}
completing the proof of \eqref{goot}. Note that the last equality
follows from our implicit identification of the Lie algebroid of $J^1
G$ with $J^1 {\mathfrak g} $, as discussed in \S\ref{jet}.
\end{proof}

\section{Integrating morphisms}\lab{morphh}

\subsection{Morphisms between pseudoactions}\lab{mk}
By a {\df morphism} of pseudoactions ${\Pi} \colon {\mathcal F}_1
\rightarrow {\mathcal F}_2$ let us mean a Lie groupoid morphism
${\Pi} \colon G_1 \rightarrow G_2$ of the underlying Lie groupoids
that maps leaves of ${\mathcal F}_1$ into leaves of ${\mathcal
  F}_2$. If, in addition, ${\Pi}$ covers a local diffeomorphism $f
\colon M_1 \rightarrow M_2$, then the restriction of ${\Pi} $ to any
leaf of ${\mathcal F}_1$ will be a local diffeomorphism into a leaf of
${\mathcal F}_2$ (because ${\Pi} $ respects source and target maps).

\begin{proposition}
  Let ${\Pi} \colon G_1 \rightarrow G_2$ be a morphism of Lie
  groupoids covering a local diffeomorphism $f \colon M_1 \rightarrow
  M_2$ and assume $G_1$ is source-connected. Let $ {\mathcal F}_1 $
  and ${\mathcal F}_2$ be pseudoactions on $G_1$ and $ G_2$
  respectively, and $\nabla^1$ and $\nabla^2$ the corresponding
  infinitesimal Cartan connections. Then ${\Pi} $ is a morphism
  ${\mathcal F}_1 \rightarrow {\mathcal F}_2$ if and only if its
  derivative $ {\pi} = d {\Pi} \colon {\mathfrak g}_1 \rightarrow
  {\mathfrak g}_2$ respects the connections $\nabla^1,\nabla^2$.
\end{proposition}
\begin{remark}
  Since $\pi \colon {\mathfrak g}_1 \rightarrow {\mathfrak g}_2$
  covers a local diffeomorphism $f \colon M_1 \rightarrow M_2$ by
  hypothesis, that $\pi $ should respect connections simply means that
  $\pi_*(\nabla^1_VX_1)=\nabla^2_{f_*V} \pi_*X_1$ for all local
  sections $X_1$ of ${\mathfrak g}_1$, and vector fields $V$ on $M_1$,
  whose domain of definition is small enough that the pushforward
  operations $\pi_*$ and $f_*$ make sense (see the proof below).
\end{remark}
\begin{proof}
  With the help of Proposition \ref{terminology} it is not hard to
  establish the following:
  \begin{lemma}
    Under the hypotheses of the proposition ${\Pi} $ is a morphism
    ${\mathcal F}_1 \rightarrow {\mathcal F}_2 $ if and only if the
    following diagram is commutative:
    \begin{equation*}
      \begin{CD}
        J^1 G_1   @>{J^1 {\Pi}  }>> J^1 G_2\\
        @A{S_1}AA @AA{S_2}A \\
        G_1@>{{\Pi} }>> G_2
      \end{CD}\quad .
    \end{equation*}
    Here $S_1$ and $S_2$ denote the corresponding (global) Cartan
    connections.
  \end{lemma}
  \noindent%
  Since ${\Pi} $ is not a base-preserving morphism, one must take
  care to interpret the map $J^1 {\Pi} \colon J^1 G_1 \rightarrow J^1
  G_2$ appropriately: Since we assume ${\Pi}$ covers a local
  diffeomorphism, each local bisection $b \colon U \rightarrow G_1$
  pushes forward to a local bisection ${\Pi}_*b := {\Pi} \circ b
  \circ f^{-1} \colon f(U) \rightarrow G_2$, provided $U$ is
  sufficiently small that $f^{-1} \colon f(U) \rightarrow U $ makes
  sense. Then $J^1 {\Pi}\, (J^1_m b):=J^1_{f(m)}({\Pi}_*b)$.

  We define a map $J^1 {\pi} \colon J^1 {\mathfrak g}_1 \rightarrow
  J^1 {\mathfrak g}_2$ similarly: Every local section $X$ of
  ${\mathfrak g}_1$, defined on a sufficiently small domain, has a
  pushforward ${\pi}_*X={\pi} \circ X \circ f^{-1}$, and we define
  $J^1 {\pi}\,(J^1_{m}X)=J^1_{f(m)}{\pi}_*X$. We leave it to the
  reader to verify that $d(J^1 {\Pi})=J^1 (d {\Pi})$, i.e., $d(J^1
  {\Pi})=J^1 {\pi} $. Then, taking derivatives, commutativity of the
  diagram in the lemma guarantees commutativity of
  \begin{equation}
      \begin{CD}
        J^1 {\mathfrak g}_1   @>{J^1 {\pi}  }>> J^1 {\mathfrak g}_2\\
        @A{s_1}AA @AA{s_2}A \\
        {\mathfrak g}_1@>{{\pi} }>> {\mathfrak g}_2
      \end{CD}\quad .\mathlab{dd}
  \end{equation}
  Here $s_1=dS_1$ and $s_2=dS_2$. In fact, commutativity of each
  diagram is equivalent to that of the other, because two Lie groupoid
  morphisms $J^1 G_1 \rightarrow G_2$ having the same derivative must
  coincide (because $G_1$, and hence $J^1 G_1$, has connected
  source-fibres). So, to prove the proposition it suffices, by the
  lemma, to show that diagram \eqref{dd} commutes if and only if
  ${\pi} $ preserves the connections $\nabla^1,\nabla^2$. Since the
  two connections $\nabla_1$ and $\nabla_2$ are flat, the latter is
  equivalent to the assertion that the pushforward ${\pi}_* X$ of
  every $\nabla^1$-parallel local section $X$ of ${\mathfrak g}_1$ is
  $\nabla^2$-parallel. That this is equivalent to commutativity of
  \eqref{dd} follows easily from the equations $s_1X = J^1 X +
  \nabla^1 X$ and $s_2 Y = J^1 Y + \nabla^2 Y$ defining the
  connections $\nabla^1$ and $\nabla^2$.
\end{proof}

\subsection{Integrating morphisms}\lab{mre}
Using the preceding proposition, we prove the following
result on integrating morphisms between twisted Lie algebra actions:
\begin{theorem}[Lie II for pseudoactions]
  Let $({\mathfrak g}_1, \nabla^1)$ and $({\mathfrak g}_2, \nabla^2)$
  be twisted Lie algebra actions formally integrated by pseudoactions
  ${\mathcal F}_1$ on $G_1$ and ${\mathcal F}_2$ on $G_2$, and assume
  $G_1$ is source-simply-connected. Suppose ${\pi} \colon {\mathfrak
    g}_1 \rightarrow {\mathfrak g}_2$ is a Lie algebroid morphism
  respecting the connections $\nabla^1,\nabla^2$, and that ${\pi} $
  covers a local diffeomorphism $f \colon M_1 \rightarrow M_2$. Then
  there exists a unique morphism ${\Pi} \colon {\mathcal F}_1
  \rightarrow {\mathcal F}_2$ whose derivative is ${\pi} $.
\end{theorem}
\begin{proof}
  By Lie II for Lie groupoid morphisms, there exists a unique Lie
  groupoid morphism ${\Pi} \colon G_1 \rightarrow G_2$ whose
  derivative is ${\pi} $. Proposition \ref{mk} implies ${\Pi} $ is a
  morphism of pseudoactions ${\mathcal F}_1 \rightarrow {\mathcal
    F}_2$.
\end{proof}

\subsection{Coverings of pseudoactions}\lab{mjy}
Let us say that a Lie groupoid $\tilde G $ over $M$ {\df covers} a Lie
groupoid $G$ over $M$ if $\tilde G$ and $G$ are source-connected and
have the same Lie algebroid ${\mathfrak g} $, and if there is a
base-preserving Lie groupoid morphism $\Pi \colon \tilde G \rightarrow
G$ whose derivative $\pi \colon {\mathfrak g} \rightarrow {\mathfrak
  g} $ is the identity.  Any covering $\Pi \colon \tilde G \rightarrow
G$ is necessarily a {\em surjective} local diffeomorphism.

A {\df covering} of a pseudoaction ${\mathcal F} $ on ${\mathcal G} $
is a covering $\Pi \colon \tilde G \rightarrow G$ of Lie groupoids,
together with a pseudoaction $\tilde {\mathcal F} $ on $\tilde G$,
such that $\Pi $ is a morphism of pseudoactions $\tilde {\mathcal F}
\rightarrow {\mathcal F} $. It follows in that case, from Proposition
\ref{mk}, that $\tilde {\mathcal F} $ and $ {\mathcal F} $ formally
integrate the same twisted Lie algebra action on $M$.
\begin{proposition}
  For any covering $\tilde {\mathcal F} \rightarrow {\mathcal F} $ of
  pseudoactions one has $\pseud(\tilde {\mathcal F})=\pseud({\mathcal
    F})$. 
\end{proposition}
\begin{proof}
  Let $\Pi \colon \tilde G \rightarrow G$ denote the underlying
  morphism of Lie groupoids. As it is more-or-less immediate that
  $\pseud(\tilde {\mathcal F})\subset \pseud({\mathcal F})$, we prove
  only the reverse inclusion. Suppose $\phi \in \pseud({\mathcal
    F})$. Then it suffices to show that each $m_0$ in the domain $U$
  of $\phi$ has a neighbourhood $V$ such that $\phi|_V \in
  \pseud(\tilde {\mathcal F})$. Shrinking $U\ni m_0$ if necessary, we
  have $\phi=\beta \circ b$, for some local bisection $ b \colon U
  \rightarrow G$ integrating ${\mathcal F}$; we are denoting both
  target projections $\tilde G \rightarrow M$ and $G \rightarrow M$ by
  $\beta $. Put $ g = b(m_0)$. Since $\Pi \colon \tilde G \rightarrow
  G$ is surjective (the crucial point) there exists $\tilde g \in
  \tilde G$ with $\Pi(\tilde g)=g$. Since $\Pi$ is morphism of
  pseudoactions $\tilde {\mathcal F} \rightarrow {\mathcal F} $, $\Pi
  $ maps the leaf $\tilde {\mathcal L} $ of $\tilde {\mathcal F} $
  through $\tilde g$ into the leaf ${\mathcal L} $ of ${\mathcal F}$
  through $g$. In fact, as $\Pi $ is base-preserving, and $\tilde
  {\mathcal L} $ and ${\mathcal L} $ are pseudotransformations, the
  restriction $\Pi \colon \tilde {\mathcal L} \rightarrow {\mathcal L}
  $ is a local diffeomorphism. There consequently exists some open
  sets $V \subset U$ and $W \subset \tilde {\mathcal L} $ such that
  $W$ is mapped diffeomorphically onto $b(V) \subset {\mathcal L} $ by
  $\Pi$, and diffeomorphically onto $V$ by the source projection. The
  local bisection $\tilde b \colon V \rightarrow \tilde G$ of $\tilde
  {\mathcal F} $ whose image is $W$ integrates $\tilde {\mathcal F} $,
  so that $\beta \circ \tilde b \in \pseud(\tilde {\mathcal F})$. But
  by construction, $\beta \circ \tilde b$ coincides with $\phi=\beta
  \circ b$ on $V$, which shows $\phi|_V \in \pseud(\tilde {\mathcal
    F})$.
\end{proof}
\begin{theorem}[Lie I for pseudoactions]
  Let $({\mathfrak g} , \nabla)$ be a twisted Lie algebra action
  formally integrated by a pseudoaction ${\mathcal F}$ on a
  source-connected Lie groupoid $G$. Then there exists a second
  pseudoaction $\tilde {\mathcal F} $ formally integrating
  $({\mathfrak g}, \nabla)$ whose underlying Lie groupoid $\tilde G$
  is source-simply-connected.  The pseudoaction $\tilde {\mathcal F} $
  is uniquely defined, up to isomorphism.
\end{theorem}
\noindent% 
We will temporarily need the following addendum (ultimately rendered
redundant by Theorem \ref{lie3} proven in the next section):
\begin{addendum}
  If additionally ${\mathcal F} $ integrates $({\mathfrak g}, \nabla)$, then we may suppose $\tilde
  {\mathcal F} $ does too. 
\end{addendum}
\begin{proof}
  By Lie I and II for Lie groupoids (see, e.g.,
  \cite{Crainic_Fernandes_11}) there exists a source-simply-connected
  Lie groupoid $\tilde G$ covering $G$, unique up to isomorphism. We
  equip $\tilde G$ with the pseudoaction whose leaves are the
  connected components of pre-images of leaves of ${\mathcal F} $
  under the covering map $\tilde G \rightarrow G$. Then $\tilde
  {\mathcal F} $ evidently covers ${\mathcal F}$ and formally
  integrates $({\mathfrak g}, \nabla )$. The addendum is a consequence
  of the preceding proposition.

  If $\tilde {\mathcal F}'$ is a second pseudoaction on $\tilde G$
  formally integrating $({\mathfrak g},\nabla)$ then the identity
  morphism on $\tilde G$ is an isomorphism of pseudoactions $\tilde
  {\mathcal F} \rightarrow \tilde {\mathcal F}' $, by Proposition
  \ref{mk}.
\end{proof}

\section{Integrating a twisted Lie algebra action}\lab{hes}
We now combine the results of \S\ref{av} and \S\ref{morphh} to prove
Theorem \ref{lie3} (Lie III for pseudoactions). Throughout
$({\mathfrak g},\nabla)$ denotes a twisted Lie algebra action on $M$.

\subsection{The integration of $({\mathfrak g},\nabla)$}\lab{yve}%
Let $(\tilde {\mathfrak g} , \tilde \nabla)$ denote the pullback of
$({\mathfrak g},\nabla)$ to the universal cover $\tilde M$ under the
covering map ${f} \colon \tilde M \rightarrow M$. Since $\tilde M $ is
simply-connected, $\tilde {\mathfrak g} $ is naturally isomorphic to
an action algebroid (by Proposition \ref{twot}, for
example). According to Proposition \ref{tuy}, there exists a
source-connected Lie groupoid $\tilde G$ supporting a pseudoaction
$\tilde {\mathcal F} $ that integrates and formally integrates
$(\tilde {\mathfrak g} , \tilde \nabla)$. The former property means
\begin{equation}
  \pseud(\tilde {\mathcal F})=\pseud(\tilde \nabla).\mathlab{hsp}
\end{equation}
By Lie I for pseudoactions (Theorem \ref{mjy}) and its Addendum, we
may take $\tilde G$ to be source-simply-connected.

Next, let $\Lambda $ denote the group of covering transformations of
$\tilde M$. Then $\Lambda $ acts canonically on $\tilde {\mathfrak g}
$ by Lie algebroid isomorphisms, and each such isomorphism respects
the connection $\tilde \nabla $. By Lie II for pseudoactions (Theorem
\ref{mre}), this action lifts to an action of $\Lambda $ on $\tilde G$
by isomorphisms of $\tilde {\mathcal F} $ (Lie groupoid isomorphisms
mapping leaves of $\tilde {\mathcal F} $ to leaves of $\tilde
{\mathcal F} $).  This action is totally discontinuous, since the
action of $\Lambda $ on $\tilde M $ is already totally
discontinuous. The quotient $G :=\tilde G/\Lambda $ is a well-defined
Lie groupoid over $M$ whose Lie algebroid is isomorphic to ${\mathfrak
  g}$. Because $\Lambda$ acts by isomorphisms of $\tilde {\mathcal
  F}$, the pseudoaction $\tilde {\mathcal F} $ drops to a foliation
${\mathcal F} $ on $G$ that is actually a pseudoaction, as is not too
hard to see. The infinitesimalization of ${\mathcal F} $ is evidently
$\nabla $.  The derivative of the covering ${\Pi} \colon \tilde G
\rightarrow G$ is the canonical projection ${\pi} \colon \tilde
{\mathfrak g} \rightarrow {\mathfrak g}$, which evidently respects the
connections $\tilde \nabla $, $\nabla $. By Proposition \ref{mk},
${\Pi} $ maps leaves of $\tilde {\mathcal F} $ into leaves of
${\mathcal F} $. In fact, since ${\Pi} $ covers a local diffeomorphism
of the base manifolds we have:
\begin{conditions}
\item\lab{dff} %
{\it The restriction
of ${\Pi} $ to any leaf of $\tilde {\mathcal F} $ is a local
diffeomorphism into the corresponding leaf of ${\mathcal F}$.}
\end{conditions}

By construction, ${\mathcal F} $ formally integrates $({\mathfrak g},
\nabla)$. We now use \eqref{hsp} and \eqref{dff} to show that
${\mathcal F} $ is a bona fide integration as well, i.e.,
$\pseud({\mathcal F})=\pseud(\nabla)$.
\subsection{Proof that $\pseud(\nabla)\subset\pseud({\mathcal F})$}\lab{txl}%
Let $X$ be a local $\nabla $-parallel section of $\mathfrak g$ and $t
\in {\mathbb R} $ such that $\Phi_{\#X}^t$ has non-empty domain
$U$. As elements of $\pseud(\nabla)$ are generated by transformations
of this form, it will suffice to prove $\Phi_{\#X}^t \in
\pseud({\mathcal F})$. Indeed, it will suffice to construct, for each
$m_0 \in U$, an open neighbourhood $V \subset U$ of $m_0$ such that
$\Phi_{\#X}^t|_V \in \pseud({\mathcal F})$.

Pull the section $X$ back to a section $\tilde X$ of $\tilde
{\mathfrak g} $, and let $\tilde m_0 \in \tilde M$ be a point covering
$m_0$. Shrinking $U \ni m_0$ if necessary, there exists a connected
neighbourhood $\tilde U$ of $\tilde m_0$ mapped diffeomorphically onto
$U$ by the projection ${f} \colon \tilde M \rightarrow M$, and we will
have $\tilde U \subset \domain \Phi^t_{\# \tilde X}$. Again supposing
$U$ to be have been chosen sufficiently small, we have, by
\eqref{hsp}, $\Phi^t_{\# \tilde X}=\beta \circ \tilde b$, for some
local bisection $\tilde b \colon \tilde U \rightarrow \tilde G$
integrating $\tilde {\mathcal F} $.  Put $\tilde g = \tilde b(\tilde
m_0)$. As the restriction of ${\Pi} $ to the leaf of $\tilde {\mathcal
  F} $ through $\tilde g$ is a local diffeomorphism into the leaf of
${\mathcal F} $ through $g={\Pi}(\tilde g)$, there exists an open
neighbourhood $\tilde V \subset \tilde U$ of $\tilde m_0$ covering
some neighbourhood $V \subset U$ of $m_0$, and a local bisection $b
\colon V \rightarrow G$ integrating ${\mathcal F} $, such that ${f}
\colon \tilde M \rightarrow M$ maps $\tilde V$ diffeomorphically onto
$V$, and
\begin{equation}
   b({f}(\tilde m))={\Pi}(\tilde b(\tilde m));\qquad \tilde m \in
    \tilde V.\mathlab{dig}
\end{equation}
But then, for any $\tilde m \in \tilde V$, we have
\begin{equation*}
  \beta(\,b({f}(\tilde m))\,)=\beta(\,{\Pi}(\tilde b(\tilde
  m))\,)={f}(\,\beta(\tilde b(\tilde m))\,)={f}(\Phi^t_{\# \tilde
    X}(\tilde m))=\Phi^t_{\# X}({f}(\tilde m)),
\end{equation*}
the last equality following from the fact that $\# \tilde X $ and $\#
X$ are ${f} $-related. This shows that $\beta \circ b
=\Phi^t_{\#X}|_V$. Since $\beta \circ b \in \pseud({\mathcal F})$, we
have $\Phi^t_{\# X}|_V \in \pseud({\mathcal F})$, as required.

\subsection{Proof that $\pseud({\mathcal F})\subset\pseud(\nabla)$}
Let $\phi \in \pseud({\mathcal F})$. To show $\phi \in \pseud(\nabla)$
is suffices to construct, for each $m_0$ in the domain $U$ of $\phi $,
an open neighbourhood $V \subset U$ of $m_0$ such that $\phi|_V \in
\pseud(\nabla)$. Shrinking $U \ni m_0$ if necessary, we have $\phi =
\beta \circ b$ for some local bisection $b \colon U \rightarrow G$
integrating ${\mathcal F}$. Let $g=b(m_0)$. Then there exists $\tilde
g \in \tilde G$ such that ${\Pi}(\tilde g)=g$, whose source will be
denoted $\tilde m_0 \in \tilde M$; we have ${f}(\tilde m_0)=m_0$. By
\eqrefs{yve}{dff}, there exists an open neighbourhood $\tilde V$ of
$\tilde m_0$, and a local bisection $\tilde b \colon \tilde V
\rightarrow \tilde G$ integrating $\tilde {\mathcal F} $, such that $f
\colon \tilde M \rightarrow M$ maps $\tilde V$ diffeomorphically onto
an open set $V \subset U$, and \eqref{dig} above holds. Since
$\tilde b$ integrates $\tilde {\mathcal F}$, it follows from
\eqrefs{yve}{hsp} (shrinking $\tilde V$ above if necessary) that there
exist $\tilde \nabla $-parallel sections $\tilde X_1,\tilde X_2,
\ldots \tilde X_k$ of $\tilde {\mathfrak g} $, and $t_1, t_2, \ldots,
t_k \in {\mathbb R}$, such that
\begin{equation}
    \beta(\tilde b(\tilde m))=\Phi^{t_1}_{\#\tilde X_1} \circ \Phi^{t_2}_{\#\tilde
      X_2} \circ \cdots \circ \Phi^{t_k}_{\#\tilde X_k}\,(\tilde
    m);\qquad \tilde m \in \tilde V.\mathlab{goolie}
\end{equation}
We may take the sections $\tilde X_j$ to be globally defined ($\tilde \nabla
$ is the canonical flat connection on a trivial bundle).

The following result, whose elementary proof is left to the reader,
sates that the image under $f \colon \tilde M \rightarrow M $ of an
integral curve of a global vector field on $\tilde W $ on $\tilde M$
can be pieced together from integral curves of locally-defined vector
fields on $M$ that are $f$-related to $\tilde W$:
\begin{lemma}
  Let $\tilde W$ be a globally defined vector field on $\tilde M$ and
  suppose $\tilde m_0 \in \domain\Phi^t_{\tilde W}$. Then there exist:
  \begin{conditions}
  \item open sets $\tilde U^1, \tilde U^2, \ldots, \tilde U^n \subset
    \tilde M$ such that $\tilde U^n \subset \domain\Phi^t_{\tilde W} $
    is a neighbourhood of $\tilde m_0$ and ${f}\colon \tilde M
    \rightarrow M$ maps $\tilde U^j$ diffeomorphically onto its image
    $U^j$,
  \item local vector fields $W^1, W^2, \ldots, W^n$ defined
    respectively on $U^1, U^2, \ldots,U^n$, and
  \item numbers $t^1, t^2, \ldots t^n \in {\mathbb R}$,
  \end{conditions}
  such that $\tilde W|_{\tilde U^i} $ is ${f}$-related to $W^i$ and
  \[ f(\Phi^t_{\tilde W}(\tilde m))=\left(\Phi_{W^1}^{t^1}\circ
    \Phi_{W^2}^{t^2} \circ \cdots \circ
    \Phi_{W^n}^{t^n}\right)\,(f(\tilde m)),\qquad \text{for all
    $\tilde m \in \tilde U^n $}.\]
\end{lemma}
\noindent%
With the help of this lemma (again shrinking $\tilde V$ if necessary)
one can find open subsets $\tilde U_i^j \subset M$ ($1 \le i \le k$,
$1\le j\le n_i$) and local sections $X_i^j$ of ${\mathfrak g} $,
defined on $U_i^j :=f(\tilde U_i^j) \subset M$, such that $\tilde
U_k^{n_k}= \tilde V$, $\pi (\tilde X_i(\tilde m))=X_i^j(f(\tilde m))$
for all $\tilde m \in \tilde U_i^j$, and
\begin{equation}
    f(\Phi^{t_i}_{\#\tilde X_i}(\tilde m))=\left(\Phi_{\#X_i^1}^{t_i^1}\circ
      \Phi_{\#X_i^2}^{t_i^2} \circ \cdots \circ
      \Phi_{\#X_i^{n_i}}^{t_i^{n_i}}\right)\,(f(\tilde m)); \qquad \tilde m \in
    \tilde U_i^n,\mathlab{jsl}
\end{equation}
for $1\le i \le k$.  

Now let $\tilde m \in \tilde V=U_k^{n_k}$ be arbitrary. Then
\begin{equation*}
    \phi(f(\tilde m))=\beta(\,b({f}(\tilde m))\,)={f}(\,\beta(\tilde b(\tilde m))\,)=
    f \Big(\,\Phi^{t_1}_{\#\tilde X_1} \circ \Phi^{t_2}_{\#\tilde
      X_2} \circ \cdots \circ \Phi^{t_k}_{\#\tilde X_k}\,(\tilde
    m)\,\Big),
\end{equation*}
where the second equality follows from \eqrefs{txl}{dig}, and the
third from \eqref{goolie}.  Finally, employing \eqref{jsl} for each
$i$ between $1$ and $k$, we obtain
\begin{multline*}
  \phi({f}(\tilde m))=\Big( \Phi_{\#X_1^1}^{t_1^1}\circ
  \Phi_{\#X_1^2}^{t_1^2} \circ \cdots \circ
  \Phi_{\#X_1^{n_1}}^{t_i^{n_1}}\\ \circ \Phi_{\#X_2^1}^{t_2^1}\circ
  \Phi_{\#X_2^2}^{t_2^2} \circ \cdots \circ
  \Phi_{\#X_i^{n_2}}^{t_2^{n_2}} \circ \cdots\\ \circ
  \Phi_{\#X_k^1}^{t_k^1}\circ \Phi_{\#X_k^2}^{t_k^2} \circ \cdots
  \circ \Phi_{\#X_k^{n_k}}^{t_i^{n_k}}\Big)(f(\tilde m)).
\end{multline*}
Since $\tilde m \in \tilde V$ is arbitrary, the formula holds with
$f(\tilde m)$ replaced with $m$, with $m \in V=f(\tilde V)$
arbitrary. We conclude $\phi|_V \in\pseud(\nabla)$.

\subsection{Proof that formal integrations are integrations} To
complete the proof of Theorem \ref{lie3} it remains to show that any
formal integration ${\mathcal F}' $ of $({\mathfrak g}, \nabla)$ is a
bona fide integration, assuming the Lie groupoid $G'$ supporting
${\mathcal F}'$ is source-connected.

Let ${\mathcal F} $ be the pseudoaction constructed in
\S\ref{yve}. Then ${\mathcal F} $ formally integrates $({\mathfrak
  g},\nabla)$ and, as we have just shown in the preceding sections,
integrates $({\mathfrak g},\nabla)$ also:
\begin{equation}
  \pseud({\mathcal F})=\pseud(\nabla).\mathlab{ovr}
\end{equation}
Moreover, ${\mathcal F} $ is supported by a Lie groupoid $G$ that is
source-simply-connected. It follows from Lie II for pseudoactions
(Theorem \ref{mre}) that there exists a Lie groupoid morphism $\Pi
\colon G \rightarrow G'$ whose derivative is the identity on
${\mathfrak g}$ and that is a morphism of pseudoactions ${\mathcal F}
\rightarrow {\mathcal F}'$, and whence a covering of pseudoactions in
the sense of \S\ref{mjy}, because we suppose $G'$ is
source-connected. Applying Proposition \ref{mjy}, we have
$\pseud({\mathcal F})=\pseud({\mathcal F}')$. Consequently,
\eqref{ovr} implies $\pseud({\mathcal F}')=\pseud(\nabla)$, i.e.,
${\mathcal F}'$ integrates $({\mathfrak g},\nabla)$.

\section{Integrating a complete twisted Lie algebra action}\lab{onu}
This section is devoted to the proof of Theorem \ref{twisted},
beginning with the untwisted case, first proven by Palais.

\subsection{Integrating a complete Lie algebra action}
Once again let ${\mathfrak g}_0 $ be a finite-dimensional Lie algebra
acting smoothly from the left on $M$, and denote the corresponding Lie
algebra homomorphism ${\mathfrak g}_0 \rightarrow \Gamma(TM)$ by $\xi
\mapsto \xi^\dagger$. The simply-connected Lie group with Lie algebra
${\mathfrak g}_0 $ will be denoted by $G_0$.  We let $\nabla $ denote
the canonical flat connection on ${\mathfrak g}_0 \times M$.

\begin{theorem}[Palais global integrability theorem \cite{Palais_57a}]
  Suppose that $\xi^\dagger$ is a complete vector field for every $\xi
  \in {\mathfrak g}_0$. Then there is a global action of $G_0$ on $M$
  whose infinitesimal generators are $\xi^\dagger$, $\xi \in
  {\mathfrak g}_0$. Moreover, if $\phi_g \colon M \rightarrow M$
  denotes the global transformation corresponding to $g \in G_0$, then
  every $\phi \in \pseud(\nabla)$ with connected domain $U$ is of the
  form $\phi=\phi_g|_U$, for some $g \in G_0$.
\end{theorem}
\noindent%
Recall here that $\pseud(\nabla)$ is the pseudogroup of
transformations generated by flows of infinitesimal generators.
\begin{proof}
  According to Lie III and Lie I for pseudoactions (Theorems
  \ref{lie3} and \ref{mjy}) there exists a pseudoaction ${\mathcal
    F}$, supported by a source-simply-connected Lie groupoid $G$, that
  both integrates and formally integrates $({\mathfrak g}_0 \times
  M,\nabla)$. We will construct a global action of $G_0$ on $M$ and a
  Lie groupoid isomorphism $G \rightarrow G_0 \times M$
  that is also an isomorphism of pseudoactions ${\mathcal F}
  \rightarrow {\mathcal F}_0$. Here $G_0 \times M$ denotes the
  corresponding action groupoid and ${\mathcal F}_0$ the canonical
  pseudoaction on $G_0 \times M$. The theorem follows readily from the
  existence of this isomorphism.

  Applying Lie II for pseudoactions (Theorem \ref{mre}), there exists
  a Lie groupoid morphism $\Omega \colon G \rightarrow G_0$
  integrating the canonical projection $\omega \colon {\mathfrak g}_0
  \times M \rightarrow {\mathfrak g}_0 $, and this morphism maps
  leaves of ${\mathcal F} $ to points. A dimension count implies that
  the leaves of ${\mathcal F}$ are connected components of fibres of
  $\Omega$. Let $P \subset G$ denote an arbitrary source-fibre of
  $G$. Then as $\omega $ is a fibre-wise isomorphism, the restriction
  $\Omega \colon P \rightarrow G_0$ is a local diffeomorphism. The
  essential point is:
  \begin{lemma}
    The restriction of $\Omega \colon G \rightarrow G_0$ to any
    source-fibre $P$ is a diffeomorphism onto $G_0$.
  \end{lemma}
  \noindent%
  Postponing the proof of the lemma, we construct an action of $G_0$
  on $M$ as follows. First note that the smooth map $\Omega \times
  \alpha \colon G\rightarrow G_0 \times M$ has full rank and so is a
  local diffeomorphism; here $\alpha$ denotes the source map. But from
  the surjectivity of $\alpha \colon G \rightarrow M$ and the lemma we
  readily see that $\Omega \times \alpha $ is bijective, and hence a
  diffeomorphism. Let $E \colon G_0 \times M \rightarrow G$ be the
  inverse diffeomorphism and --- as in the proof of Palais' local
  integrability theorem \S\ref{hgn} --- define
  $\phi_g(m)=\beta(E(g,m))$, where $\beta $ is the target map. The
  proof that this defines an action of $G_0$ on $M$ is the same as for
  the local result (read from Equation \eqrefs{hgn}{trt}, replacing
  both $Z$ and $Z_\mathrm{big}$ with $G$). This action makes $G_0
  \times M $ into a Lie groupoid; it follows by construction that the
  diffeomorphism $\Omega \times \alpha \colon G \rightarrow G_0 \times
  M$ is compatible with the respective source and target maps. That
  $\Omega \times \alpha $ is Lie groupoid morphism now follows
  immediately from the fact that $\Omega \colon G \rightarrow G_0$
  is. The derivative of $\Omega \times \alpha $ is just the identity
  on ${\mathfrak g}_0 \times M$. Since ${\mathcal F}_0$ formally
  integrates $({\mathfrak g}_0 \times M, \nabla)$, it follows from
  Proposition \ref{mk} that $\Omega$ is in fact an isomorphism of
  pseudoactions $\Omega \colon {\mathcal F} \rightarrow {\mathcal
    F}_0$.
\end{proof}

\begin{proof}[Proof of lemma]
  For each $\xi \in {\mathfrak g}_0$, let $\xi_{\mathrm c}$ denote the
  corresponding constant section of ${\mathfrak g}_0 \times M$.  Let
  $\xi_P$ denote the restriction of the corresponding right-invariant
  vector field on $G$ to one on $P$ (in earlier notation,
  $\xi_P=\xi^\mathrm{R}_\mathrm{c}|_P$). Then, by Lemma \ref{loc},
  there exists through each point $g \in P$ an integral path of
  $\xi_P$, beginning at $g$ and covering the integral path of
  $\#\xi_{\mathrm c}$ beginning at $\beta (g)$; here $\beta \colon P
  \rightarrow M$ denotes the target projection, and is the map with
  respect to which `covering' is to be understood. Since $\#
  \xi_{\mathrm c} $ is complete, by hypothesis, it follows that
  $\xi_P$ is complete also.

  We now interpret the vector field completeness in terms of
  connections. Since $\Omega \colon G \rightarrow G_0$ is a groupoid
  morphism, $\xi_P$ is $\Omega$-related to the right-invariant vector
  field on $G_0$ generated by $\xi \in {\mathfrak g}_0$. Now let
  $\nabla $ denote the canonical flat connection on $G_0$ with respect
  to which every right-invariant vector field is parallel, and pull
  $\nabla $ back to a flat connection $\nabla^P$ on $P$ using the
  local diffeomorphism $\Omega \colon P\rightarrow G_0$. Then, by
  construction, the geodesics of $\nabla^P$ are the integral curves of
  $\xi_P$, which are complete. It follow from  Proposition
  \S\ref{uniformity} given in the Appendix that the local
  diffeomorphism $\Omega \colon P \rightarrow G_0$ is a covering
  space. Since $P$ is simply-connected ($G$ is
  source-simply-connected) the lemma follows.
\end{proof}

\subsection{Proof of Theorem \ref{twisted}} We now prove our
generalization of Palais' global integrability theorem to 
twisted Lie algebra actions.

Let $({\mathfrak g},\nabla)$ denote an arbitrary twisted Lie algebra
action on $M$. According to Proposition \ref{twot}, ${\mathfrak g} $
can be identified with ${\mathfrak g}_0 \times_\mu M$, where $\mu
\colon \Lambda \rightarrow \automorphism({\mathfrak g}_0)$ is the
monodromy representation, and ${\mathfrak g}_0 $ is a Lie algebra
acting on the universal cover $\tilde M$. Since $\nabla $ is complete,
${\mathfrak g}_0$ acts on $\tilde M$ by complete vector fields, and so
Palais' global integrability theorem applies. According to the proof
of that theorem above (read with $M$ replaced with $\tilde M$), there
exists an action of $G_0$ on $\tilde M$ such that the canonical
pseudoaction $\tilde {\mathcal F}_0$ on the action groupoid $G_0
\times \tilde M$ formally integrates $({\mathfrak g}_0 \times \tilde
M, \tilde \nabla )$. Here $\tilde \nabla $ denotes the canonical flat
Cartan connection on ${\mathfrak g}_0 \times \tilde M$.

On the other hand, Remark \ref{twot} tells us that ${\mathfrak g}_0
\times \tilde M$ is naturally isomorphic to the formal pullback
$\tilde {\mathfrak g} $ of ${\mathfrak g} $, with the canonical action
of $\Lambda $ on $\tilde {\mathfrak g} $ being represented by the
action $\lambda \cdot (\xi,\tilde m)=(\mu_\lambda \xi,\lambda(\tilde
m))$. According to the constructions of \S\ref{hes}, $({\mathfrak g},
\nabla)$ is formally integrated by a pseudoaction ${\mathcal F} $ on
$G$, where $G=\tilde G/\Lambda $, $\tilde G$ is any
source-simply-connected Lie groupoid formally integrating (and hence
integrating) $\tilde {\mathfrak g} $, and the action of $\Lambda $ on
$\tilde G$ is obtained by lifting its canonical representation on
$\tilde {\mathfrak g} $ using Lie II for Lie groupoids. But according
to the preceding paragraph, we may take $\tilde G = G_0 \times \tilde
M$, in which case we see that the action of $\Lambda $ on $\tilde G =
G_0 \times \tilde M$ must be given by $\lambda \cdot (g,\tilde
m)=(\nu_\lambda (g), \lambda(m))$, where the homomorphism $\nu \colon
\Lambda \rightarrow \automorphism(G_0) $ is the lift of the monodromy
representation $\mu \colon \Lambda \rightarrow
\automorphism({\mathfrak g}_0)$ obtained by applying Lie II for Lie
groups. So we have $G = \tilde G/\Lambda =G_0\times_\nu M$.

Recall also from \S\ref{hes}, that the pseudoaction ${\mathcal F}$ on
$G$ formally integrating $({\mathfrak g},\nabla)$ is obtained by
dropping any pseudoaction $\tilde {\mathcal F} $ on $\tilde G$ that
formally integrates (and hence integrates) the pullback $(\tilde
{\mathfrak g}, \tilde \nabla )$ of $({\mathfrak g}, \nabla)$ (such a
drop being well-defined by construction). In the present case $\tilde
G =G_0 \times \tilde M$, $\tilde {\mathfrak g} = {\mathfrak g}_0
\times \tilde M$, and $\tilde \nabla $ is the canonical flat
connection; so the uniqueness part of Lie I for pseudoactions (Theorem
\ref{mjy}) implies $\tilde {\mathcal F} = \tilde {\mathcal F}_0$. It
follows that the pseudoaction on $G = \tilde G/\Lambda =G_0\times_\nu
M$ formally integrating $({\mathfrak g}, \nabla)$ is the canonical one
${\mathcal F}_0$. This completes the proof that $({\mathfrak
  g},\nabla)$ is formally integrated by a twisted Lie group action.
\noindent%

\appendix
\section{Technical details}
\subsection{Localization of multiplication in a Lie groupoid}
\lab{gpmu}%
We prove here the following result stated in \S\ref{hgn}:
\begin{proposition}
  Let $Z_\mathrm{big} \subset G$ be a paracompact, Hausdorff, open
  neighbourhood of $M$ in a Lie groupoid $G$ over $M$. Then there
  exists a neighbourhood $Z \subset Z_\mathrm{big}$ of $M$ such that
  for all $h,g \in Z$ we have $hg \in Z_\mathrm{big}$, whenever $h$
  and $g $ are multipliable.
\end{proposition}
\noindent%
In the proof we will make use of the following:
\begin{lemma}
  Let $U$ be a metric space and $\{U_i\}_{i \in I}$ an open cover of
  $U$. Then there exists a mapping $g \mapsto V_g$ of $U$ into the
  collection of open subsets of $U$ such that: (i) $g \in V_g$, and
  (ii) $V_{g_1}\cap V_{g_2}\ne \emptyset$ implies $V_{g_1}\cup V_{g_2}
  \subset U_i$, for some $i \in I$. Here $g,g_1,g_2 \in X$ are
  arbitrary.
\end{lemma}
\begin{proof}[Proof of lemma]
  Since every metric space is paracompact, we may suppose, without
  loss of generality, that the cover $\{U_i\}_{i \in I}$ is locally
  finite. Then, for each $g \in U$, we can find $\epsilon_g >0$ small
  enough that for all $i \in I$,
  \begin{equation}
    B_{3\epsilon_g}(g) \subset U_i\quad\text{whenever $g \in U_i$.}\mathlab{eld}
  \end{equation}
  Here $B_r(g)$ denotes the open ball of radius $r$ centred at $g$. To
  show the sets $V_g := B_{\epsilon_g}(g)$ meet the requirements of
  the lemma, suppose $g_0 \in V_{g_1}\cap V_{g_2}$. Suppose
  $\epsilon_{g_2}\le \epsilon_{g_1}$, and choose any $i \in I$ such
  that $g_1 \in U_i$. Then \eqref{eld} immediately implies $V_{g_1}
  \subset U_i$ and it remains to show $V_{g_2} \subset U_i$. But we
  have
  \begin{equation*}
     d(g_1,g_2)\le d(g_1,g_0)+d(g_0,g_2)\le \epsilon_{g_1} +
     \epsilon_{g_2} \le 2 \epsilon_{g_1}.    
  \end{equation*}
  So, for any $g \in V_{g_2}$, we compute
  \begin{equation*}
    d(g,g_1)\le d(g,g_2) + d(g_1,g_2)\le \epsilon_{g_2} + 2
    \epsilon_{g_1}\le 3 \epsilon_{g_1},
  \end{equation*}
  and conclude that $V_{g_2} \subset B_{3\epsilon_{g_1}}(g_1)$. It
  follows from \eqref{eld} that $V_{g_2} \subset U_i$.
\end{proof}
\begin{proof}[Proof of proposition]
  Using the continuity of multiplication, it is not hard to show that
  each point $m \in M$ has a neighbourhood $U_m$ in $Z_\mathrm{big}$
  such that $h, g \in U_m \Rightarrow hg \in Z_\mathrm{big}$, whenever
  $h$ and $g$ are multipliable. The open neighbourhood $U:=\cup_{m \in
    M} U_m $ of $M$ is a paracompact, Hausdorff, smooth manifold, and
  is therefore metrizable. Applying the preceding lemma, we
  obtain a mapping $g \mapsto V_g$ from $U$ to open subsets of $U$
  such that $g \in V_g$ and
  \begin{equation*}
    \text{$V_{g_1}\cup V_{g_2} \subset U_n$ for some $n \in M$ 
      whenever $V_{g_1}\cap V_{g_2}\ne \emptyset$,} 
  \end{equation*}
  where $g,g_1,g_2$ are arbitrary.  In particular, since $M \subset
  U$, we have $m \in V_m$ and
  \begin{equation}
    \text{$V_{m_1}\cup V_{m_2}
      \subset U_n$ for some $n \in M$ whenever $ V_{m_1}\cap V_{m_2}\ne \emptyset$,}\mathlab{jrd}
  \end{equation}
  where $m,m_1,m_2 \in M$ are arbitrary. We may also suppose that
  \begin{equation}
    \alpha(V_m) \subset V_m \cap M\quad\text{and}\quad \beta(V_m)
    \subset V_m\cap M; \qquad m \in M,\mathlab{spq}
  \end{equation}
  where $\alpha $ and $\beta $ are the source and target maps. For if
  not, replace each $V_m$ with $V_m \cap \alpha^{-1}(V_m\cap M)\cap
  \beta^{-1}(V_m \cap M)$, and \eqref{jrd} still holds.

  To show $Z :=\cup_{m \in M}V_m$ satisfies the requirements of the
  proposition, suppose $h,g \in Z$ are multipliable. Then $h \in
  V_{m_1}$ and $g \in V_{m_2}$ for some $m_1,m_2 \in M$ with
  $m:=\alpha(h)=\beta(g)$. Since $m = \alpha(h)$, \eqref{spq} implies
  $m \in V_{m_1}$; since $m = \beta(g)$, \eqref{spq} implies $m \in
  V_{m_2}$. So $V_{m_1}\cap V_{m_2}\ne\emptyset$, implying the
  existence of $n \in M$ such that $g,h \in U_n$, on account of
  \eqref{jrd}. By the definition of $U_n$, we have $hg \in Z_\mathrm{big}$.
\end{proof}

\subsection{Sufficient conditions for a local diffeomorphism to be a
  covering map}
\lab{uniformity}%
This appendix is devoted to a proof of the following:
\begin{proposition}
  Let $\phi \colon \tilde M \rightarrow M$ be a local diffeomorphism
  of smooth manifolds and suppose that $M$ is connected and
  admits a linear connection $\nabla$ such that the pullback $\tilde
  \nabla := \phi^* \nabla$ is a complete connection on $\tilde
  M$. Then $\phi \colon \tilde M \rightarrow M $ is a smooth covering
  map (and, in particular, is surjective).
\end{proposition}
\begin{lemma}[Lifting geodesics]
  Assume the hypotheses of the proposition hold. Then for any geodesic
  $\gamma \colon [a,b] \rightarrow M$ for $\nabla$, and any point
  $\tilde m \in \tilde M$ satisfying $\phi(\tilde m) = \gamma(t_0)$,
  for some $t_0 \in [a,b]$, there exists a unique geodesic $\tilde
  \gamma \colon [a,b] \rightarrow \tilde M$ for $\tilde \nabla $ such
  that $\gamma=\phi \circ \tilde \gamma $ and $\tilde
  \gamma(t_0)=\tilde m$.
\end{lemma}
\begin{proof}[Proof of lemma]
  Since $\phi$ is a local diffeomorphism, there exits a unique tangent
  vector $v \in T_{\tilde m}\tilde M $ such that by $T \phi \cdot v =
  \dot \gamma(t_0)$. Let $\tilde \gamma \colon [a,b] \rightarrow
  \tilde M$ be the unique geodesic for $\tilde \nabla $ such that
  $\dot {\tilde \gamma}(t_0)=v$, which exists because $\tilde \nabla $
  is complete. Then $\gamma' := \phi \circ \tilde \gamma $ must be a
  geodesic for $\nabla $ and satisfy $\dot \gamma'(t_0)=\dot
  \gamma(t_0)$. By the uniqueness of geodesics with prescribed
  velocity at a point, we have $\gamma'=\gamma$, so that $\tilde
  \gamma $ has the desired properties.
\end{proof}
\begin{proof}[Proof of proposition]
  We show that $\phi $ evenly covers normal neighbourhoods for the
  connection $\nabla $. Indeed, let $m_0 \in M$ be arbitrary and let
  $B \subset M$ be a normal neighbourhood about $m_0$ for
  $\nabla$. That is, $B$ is the diffeomorphic image of some open set
  $U \subset T_{m_0}M$ under the map $\exp_{m_0}\colon U \rightarrow
  M$ well-defined by the requirement that $\gamma(t):= \exp_{m_0}(tv)$
  be a geodesic for $\nabla $ satisfying $\gamma(0)=m_0$ and $\dot
  \gamma (0)=v$, for each choice of $v \in T_{m_0}M$.

  We first suppose $\phi^{-1}(B)$ is non-empty. Later, we show that
  $\phi$ is surjective, so that this is no restriction. Let $\tilde B$
  be any connected component of $\phi^{-1}(B)$. We must show the
  restriction $\phi \colon \tilde B \rightarrow B$ is a
  homeomorphism. To this end, consider the discrete set $X:=\tilde B
  \cap \phi^{-1}(m_0)$. To see that $X$ is non-empty let $\tilde m$ be
  any point in $\tilde B$ (non-empty by hypothesis) and let $\gamma
  \colon [0,1] \rightarrow M$ be the geodesic joining $m_0 \in B$ to
  $\phi(\tilde m) \in B$. Then if $\tilde \gamma \colon [0,1]
  \rightarrow \tilde M$ is the lift of $\gamma $ which ends at $\tilde
  m$, whose existence is guaranteed by the lemma, then $\tilde
  \gamma(0)\in X$, because $\tilde \gamma^{-1}(0)$ and $\tilde
  \gamma^{-1}(1)$ must lie in the same connected component of $\phi
  ^{-1}(B)$.

  In fact $X$ has just one element. To see this construct, for each $x
  \in X$, a smooth right-inverse $s_x \colon B \rightarrow \tilde B$
  for the restriction $\phi \colon \tilde B \rightarrow B$ as follows:
  If $m \in B$ is arbitrary and $\gamma \colon [0,1] \rightarrow B$ is
  the unique geodesic joining $m_0$ to $m$, then $s_x(m):=\tilde
  \gamma(1)$, where $\tilde \gamma \colon [0,1] \rightarrow \tilde B$
  is the lift of geodesic $\gamma$ satisfying $\tilde \gamma
  (0)=x$. Being a right-inverse, the derivative of $s_x$ has full
  rank, i.e., is a local diffeomorphism. In particular, its image
  $s_x(B)$ is an open subset of $\tilde B$. It is easy to see that the
  open sets $s_x(B)$, $x \in X$, are disjoint (use the uniqueness of
  geodesic lifts with prescribed end-point) and cover $\tilde B$. But
  $\tilde B$ is {\em connected}, implying $X$ has a single element,
  $x_0$, say. Moreover, the single right-inverse $s_{x_0}\colon B
  \rightarrow \tilde B$ must be surjective.

  Being surjective, the smooth right-inverse $s_{x_0}\colon B
  \rightarrow \tilde B$ is a two-sided inverse for the restriction
  $\phi \colon \tilde B \rightarrow B$, which is accordingly a
  homeomorphism.

  To summarise, every normal neighbourhood $B \subset M$ for $\nabla $
  is evenly covered by $\phi$ whenever it has non-trivial intersection
  with the image of $\phi$. In particular, if a point $m_0 \in M$ is
  {\em not} in this image, then any normal neighbourhood of $m_0$ must
  intersect $\phi(\tilde M)$ in an empty set. So the complement of
  $\phi(\tilde M)$ is open in $M$. Since $\phi(\tilde M) $ is itself
  open ($\phi$ is a local diffeomorphism) and $M$ is connected, we
  conclude that $\phi$ is surjective.
\end{proof}

%\bibliography{dynamics2}

\end{document}